\documentclass[12pt]{amsart}
\usepackage{amsmath}
\usepackage{amsthm}
\usepackage{amssymb}
\usepackage{mathrsfs}
\usepackage{verbatim}
\usepackage{hyperref}
\usepackage{color}

\begin{document}

\newtheorem{theorem}{Theorem}    
\newtheorem{proposition}[theorem]{Proposition}
\newtheorem{conjecture}[theorem]{Conjecture}
\def\theconjecture{\unskip}
\newtheorem{corollary}[theorem]{Corollary}
\newtheorem{lemma}[theorem]{Lemma}
\newtheorem{sublemma}[theorem]{Sublemma}
\newtheorem{observation}[theorem]{Observation}
\newtheorem{remark}[theorem]{Remark}
\newtheorem{definition}[theorem]{Definition}
\theoremstyle{definition}
\newtheorem{notation}[theorem]{Notation}
\newtheorem{question}[theorem]{Question}
\newtheorem{questions}[theorem]{Questions}
\newtheorem{example}[theorem]{Example}
\newtheorem{problem}[theorem]{Problem}
\newtheorem{exercise}[theorem]{Exercise}

\numberwithin{theorem}{section} \numberwithin{theorem}{section}
\numberwithin{equation}{section}

\def\earrow{{\mathbf e}}
\def\rarrow{{\mathbf r}}
\def\uarrow{{\mathbf u}}
\def\varrow{{\mathbf V}}
\def\tpar{T_{\rm par}}
\def\apar{A_{\rm par}}

\def\reals{{\mathbb R}}
\def\torus{{\mathbb T}}
\def\heis{{\mathbb H}}
\def\integers{{\mathbb Z}}
\def\naturals{{\mathbb N}}
\def\complex{{\mathbb C}\/}
\def\distance{\operatorname{distance}\,}
\def\support{\operatorname{support}\,}
\def\dist{\operatorname{dist}\,}
\def\Span{\operatorname{span}\,}
\def\degree{\operatorname{degree}\,}
\def\kernel{\operatorname{kernel}\,}
\def\dim{\operatorname{dim}\,}
\def\codim{\operatorname{codim}}
\def\trace{\operatorname{trace\,}}
\def\Span{\operatorname{span}\,}
\def\dimension{\operatorname{dimension}\,}
\def\codimension{\operatorname{codimension}\,}
\def\nullspace{\scriptk}
\def\kernel{\operatorname{Ker}}
\def\ZZ{ {\mathbb Z} }
\def\p{\partial}
\def\rp{{ ^{-1} }}
\def\Re{\operatorname{Re\,} }
\def\Im{\operatorname{Im\,} }
\def\ov{\overline}
\def\eps{\varepsilon}
\def\lt{L^2}
\def\diver{\operatorname{div}}
\def\curl{\operatorname{curl}}
\def\etta{\eta}
\newcommand{\norm}[1]{ \|  #1 \|}
\def\expect{\mathbb E}
\def\bull{$\bullet$\ }
\def\C{\mathbb{C}}
\def\R{\mathbb{R}}
\def\Rn{{\mathbb{R}^n}}
\def\Sn{{{S}^{n-1}}}
\def\M{\mathbb{M}}
\def\N{\mathbb{N}}
\def\Q{{\mathbb{Q}}}
\def\Z{\mathbb{Z}}
\def\F{\mathcal{F}}
\def\L{\mathcal{L}}
\def\S{\mathcal{S}}
\def\supp{\operatorname{supp}}
\def\dist{\operatorname{dist}}
\def\essi{\operatornamewithlimits{ess\,inf}}
\def\esss{\operatornamewithlimits{ess\,sup}}
\def\xone{x_1}
\def\xtwo{x_2}
\def\xq{x_2+x_1^2}
\def\Oz{\Omega}
\def\oz{\omega}
\newcommand{\abr}[1]{ \langle  #1 \rangle}

\newcommand{\Norm}[1]{ \left\|  #1 \right\| }
\newcommand{\set}[1]{ \left\{ #1 \right\} }
\def\one{\mathbf 1}
\def\whole{\mathbf V}
\newcommand{\modulo}[2]{[#1]_{#2}}
\renewcommand{\thefootnote}{\fnsymbol{footnote}}
\def\scriptf{{\mathcal F}}
\def\scriptg{{\mathcal G}}
\def\scriptm{{\mathcal M}}
\def\scriptb{{\mathcal B}}
\def\scriptc{{\mathcal C}}
\def\scriptt{{\mathcal T}}
\def\scripti{{\mathcal I}}
\def\scripte{{\mathcal E}}
\def\scriptv{{\mathcal V}}
\def\scriptw{{\mathcal W}}
\def\scriptu{{\mathcal U}}
\def\scriptS{{\mathcal S}}
\def\scripta{{\mathcal A}}
\def\scriptr{{\mathcal R}}
\def\scripto{{\mathcal O}}
\def\scripth{{\mathcal H}}
\def\scriptd{{\mathcal D}}
\def\scriptl{{\mathcal L}}
\def\scriptn{{\mathcal N}}
\def\scriptp{{\mathcal P}}
\def\scriptk{{\mathcal K}}
\def\frakv{{\mathfrak V}}

\allowdisplaybreaks

\arraycolsep=1pt
%
\newtheorem*{remark0}{\indent\sc Remark}
%
\renewcommand{\proofname}{Proof.} 

\title[Uniform weighted bounds for fractional Marcinkiewicz integrals.]
{The uniform quantitative weighted bounds for fractional type Marcinkiewicz integrals and commutators. }

\author{ Huoxiong Wu and Lin Wu$^*$}

\footnote[0]{{2020} \textit{Mathematics Subject Classification}. 42B20, 42B25}


%
\thanks{Supported by the NNSF of China (Nos. 12171399, 11871101)}
\thanks{$^*$Corresponding author.}
\address{School of Mathematical Sciences, Xiamen University, Xiamen 361005, China} \email{huoxwu@xmu.edu.cn}
\address{School of Mathematical Sciences, Xiamen University, Xiamen 361005, China} \email{wulin@stu.xmu.edu.cn}



\begin{abstract}
 In this paper we consider the fractional type Marcinkiewicz integral operator
$$\mu_{\Omega ,\beta }f(x) = \left ( \int_{0}^{\infty } \left | \int_{\left | x-y \right |\le t }^{}  \frac{\Omega (x-y)}{\left | x-y \right |^{n-1-\beta }  } f(y)dy\right | ^{2}\frac{dt}{t^3} \right )^{{1}/{2} },\quad 0<\beta<n,$$
and the corresponding commutator $\mu_{\Omega,\beta}^b$ generated by $\mu_{\Omega,\beta}$ with $b\in BMO(\mathbb{R}^n)$. Generally, the bounds of $\mu_{\Omega,\beta}$ and $\mu_{\Omega,\beta}^b$ depend on the parameter $\beta$. This paper gives the uniform quantitative weighted bounds for $\mu_{\Omega,\beta}$ and $\mu_{\Omega,\beta}^b$ about $\beta$ on the weighted Lebesgue spaces. Moreover, the corresponding bounds for the classical Marcnkiewicz integral $\mu_\Omega$ and the commutator $\mu_\Omega^b$ can be recovered from ones of $\mu_{\Omega,\beta}$ and $\mu_{\Omega,\beta}^b$ when $\beta\to 0^+$.
\end{abstract}

\maketitle

\section{Introduction}

Let $\mathbb{R}^n$, $n\ge 2$, be the $n$-dimensional Euclidean space and $\mathbb{S}^{n-1}$ be the unit sphere of $\mathbb{R}^n\,(n\ge 2)$ equipped with the normalized Lebesgue measure $d\sigma=d\sigma(x')$. And let $\Omega \in L^1({\mathbb{S}}^{n-1})$ be homogeneous of degree zero and satisfy
\begin{equation}\label{1.1}
\int_{\mathbb{S}^{n-1}}\Omega(x')d\sigma(x')=0,	
\end{equation}
where $x'=x/|x|$ for any $x\neq 0$. For $0<\beta<n$, we consider the fractional type Marcinkiewicz integrals
\begin{align*}
\mu_{\Omega ,\beta }f(x) = \left ( \int_{0}^{\infty } \left | \int_{\left | x-y \right |\le t }^{}  \frac{\Omega (x-y)}{\left | x-y \right |^{n-1-\beta }  } f(y)dy   \right | ^{2}\frac{dt}{t^3}   \right )^{1/2 },
\end{align*}
and the corresponding commutators
\begin{align*}
\mu_{\Omega ,\beta }^b f(x) = \left ( \int_{0}^{\infty } \left | \int_{\left | x-y \right |\le t }^{}  [b(x)-b(y)]\frac{\Omega (x-y)}{\left | x-y \right |^{n-1-\beta }  } f(y)dy   \right | ^{2}\frac{dt}{t^3}   \right )^{1/2 },
\end{align*}
where $b \in BMO(\mathbb{R}^n)$ with
$$\|b\|_*=:\sup_Q \frac{1}{|Q|}\int_Q |b(y)-b_Q|dy <\infty,\,\,b_Q=\frac{1}{|Q|}\int_Q b(z)dz .$$
When $\beta=0$, we denote $\mu_{\Omega,0}$ by $\mu_{\Omega}$, and $\mu_{\Omega, 0}^b$ by $\mu_\Omega^b$.

It is well known that $\mu_\Omega$ is the classical Marcinkiewicz integral, which was introduced by Stein in \cite{Ste} and was extensively studied, for example, see \cite{Als, BCP,TW} etc. In particular, Hu and Qu \cite{HQ}
obtained the following quantitative weighted result:
\begin{equation}\label{eq1.2}
\|\mu_\Omega f\|_{L^p(\omega)}\lesssim \|\Omega\|_{L^q(\mathbb{S}^{n-1})}[\omega]_{A_{p/q'}}^{\max\{1,\frac {q'}{p-q'}\}+\max\{\frac 12,\frac 1{p-q'}\}}\|f\|_{L^p(\omega)},\,\,\forall\, f\in L^p(\omega)
\end{equation}
for $\Omega\in L^q(\mathbb{S}^{n-1})$ with some $q>1$, $q'<p<\infty$ and $\omega\in A_{p/q'}$. Here and below, we say that $\omega\in A_p$ for $1<p<\infty$ (see \cite{MW}) if
$$[\omega]_{p}:= \sup_{Q}\Big(\frac{1}{|Q|}\int_{Q}w\Big)\Big(\frac{1}{|Q|}\int_{Q}\omega^{-\frac{1}{p-1}}\Big)^{p-1}<\infty, $$
and $A_{\infty}:=\cup_{p \ge 1}A_p$. Moreover, for $1<p,\, q<\infty$, we say $\omega\in A_{p,q}$ if
$$[\omega]_{A_{p,q}}:=\sup_{Q}\Big(\frac{1}{|Q|}\int_{Q}\omega^q\Big)\Big(\frac{1}{|Q|}\int_{Q}\omega^{-p'}\Big)^{q/p'}<\infty.$$
And it is easy to check that
$$[\omega^q]_{1+q/p'}=[\omega]_{A_{p,q}}\,\, \mathrm{and }\,\, [\omega^{-p'}]_{1+p'/q}=[\omega]_{A_{p,q}}^{p'/q}.$$

Meanwhile, the commutator $\mu_\Omega^b$ was first studied by Tochinsky and Wang in \cite{TW} and subsequently, attracted many researchers' attentions, see \cite{DLY, HY,CD,WW,WenW} etc. and therein references. Especially, Wen and Wu \cite{WenW} showed that for $b\in BMO(\mathbb{R}^n)$, $\Omega\in L^q(\mathbb{S}^{n-1})$ with $q>1$, $q'<p<\infty$ and $\omega\in A_{p/q'}$,
\begin{equation}\label{eq1.3}
\|\mu_\Omega^b f\|_{L^p(\omega)}\lesssim \|\Omega\|_{L^q(\mathbb{S}^{n-1})}\|b\|_*[\omega]_{A_{p/q'}}^{2\max\{1,\frac {q'}{p-q'}\}+\max\{\frac 12,\frac 1{p-q'}\}}\|f\|_{L^p(\omega)},\,\,\forall\, f\in L^p(\omega).
\end{equation}

As a natural extension $\mu_{\Omega,\beta}$ and $\mu_{\Omega,\beta}^b$ for $0<\beta<n$ have also been paid attentions by several authors. For the unweighted relevant results of $\mu_{\Omega,\beta}$ and $\mu_{\Omega,\beta}^b$, we refer to  \cite{Als,SY,SWJ,XYY1,XYY2,WW} etc. In this paper, we will focus on the corresponding quantitative weighted bounds of $\mu_{\Omega,\beta}$ and $\mu_{\Omega,\beta}^b$ for $0<\beta<n$.

 To do this, we first recall the definitions of sparse operators. Let $\eta \in (0,1)$ and $\mathcal{S}$ be a family of cubes. We say that $\mathcal{S}$ is $\eta$-sparse, if for each fixed $Q \in \mathcal{S}$, there exists a measurable subset $E_Q \subset Q$, such that $|E_Q|\geq \eta |Q| $
 and $\{E_Q\}$ are pairwise disjoint. Associated with the sparse family $\mathcal{S}$ and $r \in (0,+\infty)$, we define the fractional sparse operator $\mathcal{A}_{\mathcal{S}}^{r,\beta}$ by
 $$\mathcal{A}_{\mathcal{S}}^{r,\beta} f(x)=\Big\{\sum_{Q \in \mathcal{S}}(|Q|^{\frac{\beta}{n}}\langle|f|\rangle_{Q})^r\chi_{Q}(x) \Big\}^{1/r}, \quad 0<\beta<n,$$
and $\langle|f|\rangle_{Q} =|Q|^{-1}\int_Q |f(y)|dy$. For simplicity, we use denote $\mathcal{A}_{\mathcal{S}}^{1,\beta}$ by $\mathcal{A}_{\mathcal{S}}^{\beta}$. In \cite{IRV}, Iba\~nez-Firnkorn et al. (see also \cite{FH}) obtained the following result.

\medskip

\noindent \textbf{Theorem A} (cf.\cite{FH,IRV})
{\it Let $0<\beta<n, 1<p<n/\beta$ and $1/q=1/p-\beta/n$. If $\omega \in A_{p,q}$, then
\begin{equation}\label{eq1.4}
\|\mathcal{A}_{\mathcal{S}}^{\beta} f\|_{L^q(\omega^q)} \leq C(n) [\omega]_{A_{p,q}}^{\max\{\frac{p'}{q}(1-\frac{\beta}{n}), 1-\frac{\beta}{n}\}}\|f\|_{L^p(\omega^p)}.
\end{equation}	
}

 \medskip

On the other hand, by Minkowski's inequality, it is not hard to check that for $\Omega\in L^\infty(\mathbb{S}^{n-1})$,
\begin{equation}\label{eq1.5}
\mu_{\Omega,\beta} f(x)\le \|\Omega\|_{\infty}I_\beta(|f|)(x),
\end{equation}
where
$$I_\beta f(x)=\int_{\mathbb{R}^n}\frac{f(y)}{|x-y|^{n-\beta}}dy.$$
Also, it follows from \cite{CM} that for every bounded $f\ge 0$ with compact support, there exist a sparse family depending on $f$, $\mathcal{S}=\mathcal{S}(f)$, and constants $C_1$, $C_2$ only depending on $n$, such that
\begin{equation}\label{eq1.6}
C_1\mathcal{A}_{\mathcal{S}}^\beta f(x)\le I_\beta f(x)\le \frac {C_2}{1-2^{-\beta}}\mathcal{A}_{\mathcal{S}}^\beta f(x),\,\,\forall\, 0<\beta<n.
\end{equation}
Based on the estimates (\ref{eq1.4})-(\ref{eq1.6}), ones can obtain that for $\Omega\in L^\infty(\mathbb{S}^{n-1})$, $0<\beta<n,\, 1<p<q<\infty$ with $1/q=1/p-\beta/n$, and $\omega \in A_{p,q}$,
\begin{equation}\label{eq1.7}
\|\mu_{\Omega,\beta}f\|_{L^q(\omega^q)} \leq\frac{ C(n)}{1-2^{-\beta}}\|\Omega\|_\infty [\omega]_{A_{p,q}}^{\max\{\frac{p'}{q}(1-\frac{\beta}{n}), 1-\frac{\beta}{n}\}}\|f\|_{L^p(\omega^p)}.
\end{equation}	
Obviously, when $\beta\to 0^+$, the right side (\ref{eq1.7}) tends to infinity. However, for the un-weighted case, we \cite{WW} recently obtain the following result

\medskip

\noindent {\bf Theorem B} (cf. \cite{WW})
{\it Let $1<q<\infty$ and $0<\beta<\frac{(q-1)n}{q}.$ Suppose that $\Omega \in L^{\infty}(\mathbb{S} ^{n-1})$ is a homogeneous function of degree zero on $\mathbb{R}^n$ satisfying (\ref{1.1}). Then there exists a constant $C$ only depending on $n$, $q$ and $\Omega $ such that
\begin{align*}
\left \| \mu _{\Omega ,\beta }f  \right \| _{q}\le C\left(\left \| f \right \|_{q}+\frac{\beta ^{\frac{(q-1)n}{q} } }{\sqrt[q]{(n(q-1))-\beta q} }\left \| f \right \|_{1} \right),
\end{align*}
and for any $\lambda>0$,
\begin{equation*}
\left | \left \{ x\in \R^{n}:\mu _{\Omega ,\beta }f(x)>\lambda    \right \}  \right | \le C\left(\frac{\left \| f \right \|_{1}}{\lambda }  +\frac{\beta ^{n(q-1)} }{{n(q-1)-\beta q} }\frac{\left \| f \right \|_{1}^{q} }{\lambda ^{q} }  \right).
\end{equation*}
}

\medskip
Clearly, the bounds of $\mu_\Omega$ can be recovered from the above estimates of $\mu_{\Omega,\beta}$ when $\beta\to 0^+$. Therefore, for the weighted cases, it is natural ask the following question:

\smallskip

{\bf Question:} Can we establish the uniform weighted bounds of $\mu_{\Omega,\beta}$ for $0<\beta<n$ small enough?  As the same for $\mu_{\Omega,\beta}^b$?

\smallskip

The purpose in this paper is to address the question above. Our main results can be formulated as follows.

\begin{theorem}\label{thm1}
Let $0<\beta <n,\, 1< p<q<\infty$ with ${1}/{q}={1}/{p}-{\beta}/{n}$. Suppose that $\Omega \in L^{\infty}(\mathbb{S}^{n-1}) $ is a homogeneous function of degree zero on $\mathbb{R}^n$ satisfying (\ref{1.1}). Then for $\omega\in A_{p,q}$, 

{\rm(i)} when $0<\beta<1/2$,
\begin{equation}\label{eq1.8}
\|\mu_{\Omega,\beta}f\|_{L^q(\omega^q)} \leq C(n,p,q)\|\Omega\|_{\infty}[\omega]_{A_{p,q}}^{\max \{1,\frac{p'}{q}\}}[\omega]_{A_{p,q}}^{\max\{\frac{p'}{q}(1-\frac{\beta}{n}), \frac{1}{2}-\frac{\beta}{n}\}}\|f\|_{L^p(\omega^p)};
\end{equation}

{\rm(ii)} when $1/2\le \beta<n$,
\begin{equation}\label{eq1.9}
\|\mu_{\Omega,\beta}f\|_{L^q(\omega^q)} \leq C(n,p,q)\|\Omega\|_{\infty}[\omega]_{A_{p,q}}^{\max\{\frac{p'}{q}(1-\frac{\beta}{n}), 1-\frac{\beta}{n}\}}\|f\|_{L^p(\omega^p)}.
\end{equation}
\end{theorem}

\begin{theorem}\label{thm2}
Let $0<\beta <n$ and $1< p<q<\infty$ with ${1}/{q}={1}/{p}-{\beta}/{n}$, $b \in BMO(\mathbb{R}^n)$. Suppose that $\Omega \in L^{\infty}(\mathbb{S}^{n-1}) $ is a homogeneous function of degree zero on $\mathbb{R}^n$ satisfying (\ref{1.1}). Then for $\omega \in A_{p,q}$, there exists an absolute constant $C=C(n,p)$ depending on $n,\,p$ such that

{\rm (i)} when $0<\beta< 1/2$,
\begin{equation}\label{eq1.10}
\| \mu_{\Omega,\beta}^b f\|_{L^q(\omega^q)} \leq C(n,p,q)\|\Omega\|_{\infty} \|b\|_*[\omega]_{A_{p,q}}^{1+\max\{1,\frac{p'}{q}\}}	[\omega]_{A_{p,q}}^{\max\{\frac{p'}{q}(1-\frac{\beta}{n}),\frac{1}{2}-\frac{\beta}{n}\}}\|f\|_{L^p(\omega^p)};
\end{equation}

{\rm(ii)} when $1/2\le\beta<n$,
\begin{equation}\label{eq1.11}
\| \mu_{\Omega,\beta}^b f\|_{L^q(\omega^q)} \leq C(n,p,q)\|\Omega\|_{\infty} \|b\|_*[\omega]_{A_{p,q}}^{1+\max\{\frac{p'}{q}(1-\frac{\beta}{n}),1-\frac{\beta}{n}\}}\|f\|_{L^p(\omega^p)}.
\end{equation}
\end{theorem}

\begin{remark} We remark that when $\beta\to 0^+$, the corresponding results for $\mu_\Omega$ in \cite{HQ} and $\mu_\Omega^b$ in \cite{WenW} can be recovered. Therefore, our results can be regarded as the generalization of ones in \cite{HQ,WenW}. The main ingredient of this paper is to establish the sparse domination. 
\end{remark}

The rest of this paper is organized as follows. In Section 2, we will give a standard decomposition of $\mu_{\Omega,\beta}$ and several auxiliary lemmas, which will play key roles in the arguments later. And then we will establish the sparse domination in Section 3. Finally, the proofs of Theorems \ref{thm1} and \ref{thm2} will be presented in Section 4.

{\bf Notation.} For any cube $Q$ in $\mathbb{R}^n$ and $a>0$, we denote $aQ$ the cube with same center as $Q$ and the side-length $a\ell(Q)$, where $\ell(Q)$ is the side-length of $Q$. Throughout this paper, all the cubes are open. For a measurable set $E$ in $\mathbb{R}^n$, $|E|$ means its Lebesgue measure. The letter $C$, sometimes with additional parameters, will stand for positive constants, not necessarily the same one at each occurrence, but independent of the essential variables. By $A\sim B$, we mean that $A$ is equivalent to $B$, that is, there exist two positive constants $c$ and $C$ such that $cA\leq B\leq CA$. For $f \in L^q(\mathbb{R}^n)$ with $1\leq q\leq \infty$, we denote its $L^{q}$ norm by $\|f\|_q$. We denote by $\hat{f}$ the Fourier transform of $f$, which is defined by $\hat{f}(\xi)=\int_{\mathbb{R}^n}e^{-2\pi i\langle x,\xi  \rangle}f(x)dx.$

\section{Preliminaries}
 In this section, we first give a standard decomposition of $\mu_{\Omega,\beta}$ and some auxiliary operators, and then present several lemmas, which are the preliminaries to establish the sparse domination of $\mu_{\Omega,\beta}$ in Section 3.

  \smallskip

 \subsection{The decomposition of $\mu_{\Omega,\beta}$ and auxiliary operators}

\smallskip

For simplicity, we let $\psi (x):= |x|^{-n+\beta+1} \Omega (x') \chi_{(0,1]}(x)$ and $\psi_t(x):=\frac{1}{t^n}\psi(\frac{x}{t})$ for any $t>0$. 
And we write
\begin{align*}
\psi_t(x)& =\sum_{k<0}2^{k(\beta+1)}\left( \frac{1}{t^{\beta+1}}	|x|^{-n+\beta+1} \Omega(x') \chi_{(1,2]}(\frac{2^{-k}|x|}{t}) 2^{-k(\beta+1)}\right)\\
&=: \sum_{k<0}2^{k(\beta+1)}\psi_t^{(k)}(x).
\end{align*}
Then, by Minkowski's inequality and variable changes, we have
\begin{align*}
\mu_{\Omega,\beta}f(x)&=\left(\int_{0}^{+\infty}\left|t^{\beta}\psi_{t}\ast f(x) \right|^2\frac{dt}{t}\right)^{1/2}\\
& = 	\left(\int_{0}^{+\infty}\left| \sum_{k<0} 2^{k(\beta+1)} t^{\beta}\psi_{t}^{(k)}\ast f(x) \right|^2\frac{dt}{t}\right)^{1/2}\\
& \leq \sum_{k<0} 2^{k(\beta+1)}	\left(\int_{0}^{+\infty}\left|t^{\beta}\psi_{t}^{(k)}\ast f(x) \right|^2\frac{dt}{t}\right)^{1/2}\\
& = \sum_{k<0} 2^{k}	\left(\int_{0}^{+\infty}\left|t^{\beta}\psi_{2^{-k}t}^{(k)}\ast f(x) \right|^2\frac{dt}{t}\right)^{1/2}\\
& = \sum_{k<0} 2^{k}	\left(\int_{0}^{+\infty}\left|t^{\beta}\psi_{t}^{(0)}\ast f(x) \right|^2\frac{dt}{t}\right)^{1/2}\\
& \leq \left(\int_{0}^{+\infty}\left|t^{\beta}\psi_{t}^{(0)}\ast f(x) \right|^2\frac{dt}{t}\right)^{1/2}.
\end{align*}
Thus,
\begin{align*}
\mu_{\Omega,\beta}f(x)& \leq \left ( \int_{0}^{\infty } \left | \frac{1}{t}\int_{t \leq \left | y \right |\le 2t }^{}  \frac{\Omega (y')}{\left | y \right |^{n-1-\beta }  } f(x-y)dy   \right | ^{2}\frac{dt}{t}   \right )^{1/2 }\\
&= \left ( \sum_{j \in \mathbb{Z}}\int_{2^j}^{2^{j+1} } \left | \frac{1}{t}\int_{t \leq \left | y \right |\le 2t }  \frac{\Omega (y')}{\left | y \right |^{n-1-\beta }  } f(x-y)dy   \right | ^{2}\frac{dt}{t}   \right )^{1/2 }\\
&= \left ( \sum_{j \in \mathbb{Z}}\int_{1}^{2 } \left | \frac{1}{2^j t}\int_{2^j t \leq \left | y \right |\le 2^{j+1}t }  \frac{\Omega (y')}{\left | y \right |^{n-1-\beta }  } f(x-y)dy   \right | ^{2}\frac{dt}{t}   \right )^{1/2 }\\
&=: \left ( \sum_{j \in \mathbb{Z}}\int_{1}^{2 } \left | [T_{\beta,j}]_t f(x)\right | ^{2}\frac{dt}{t}   \right )^{1/2 }.
\end{align*}
Set
$$\widetilde{\mu}_{\Omega,\beta}f(x):= \left ( \sum_{j \in \mathbb{Z}}\int_{1}^{2 } \left | [T_{\beta,j}]_t f(x)\right | ^{2}\frac{dt}{t}   \right )^{1/2 }.
$$
It is easy to check that
$${\mu}_{\Omega,\beta}f(x)\approx \widetilde{\mu}_{\Omega,\beta}f(x) ,$$
and
$${\mu}_{\Omega,\beta}^bf(x)\approx \widetilde{\mu}_{\Omega,\beta}^b f(x).$$

Let $\varphi \in C_{0}^{\infty}(\mathbb{R}^n)$ be a nonnegative function such that $\int_{\mathbb{R}^n} \varphi(x)dx=1$, $\supp \varphi \subset \{x:|x|\leq 1/4\}$. For $l \in \mathbb{Z}$, let $\varphi_l(y)=2^{-nl}\varphi(2^{-l}y)$. It is easy to verify that
\begin{equation}\label{eq2.1}
\vert \widehat{\varphi_{j-l}}	(\xi)-1\vert \lesssim \min\{1,|2^{j-l}\xi|\},\quad\forall\, j\ ,l\in\mathbb{Z}.
\end{equation}
Let
$$
 \lbrack F_{\beta,j}^l\rbrack_t f(x):=\int_{\mathbb{R}^n}[K_{\beta,j}]_t \ast \varphi_{j-l}(x-y)f(y)dy.
$$
Define the operator $\widetilde{\mu}_{\Omega,\beta}^l$ for every $l\in\mathbb{Z}$ by
$$
\widetilde{\mu}_{\Omega,\beta}^l f(x):=  \left ( \sum_{j \in \mathbb{Z}}\int_{1}^{2 } \left | \lbrack F_{\beta,j}^l\rbrack_t f(x)\right | ^{2}\frac{dt}{t}   \right )^{1/2 }.
$$
And the grand maximal operator $\mathcal{M}_{\widetilde{\mu}_{\Omega,\beta}^l} $ are defined as follows.,
$$\mathcal{M}_{\widetilde{\mu}_{\Omega,\beta}^l}f(x):=\sup_{Q \ni x}ess \sup_{\xi \in Q} | \widetilde{\mu}_{\Omega,\beta}^l (f\chi_{\mathbb{R}^n\setminus 3Q})(\xi)|.$$

\medskip

\subsection{Some lemmas}

\begin{lemma}{\rm (\cite{LDY})}\label{lem2.1}
For $0<\beta<n,1<p,q<\infty$ and $1/q=1/p-\beta/n$,
$$\|\widetilde{I}_\beta f\|_q \leq \left(\omega_{n-1}+\left(\frac{q \omega_{n-1}}{p' n}\right)^{1/p'}\beta\right) 2^{-\beta} \pi^{-\frac{n}{2}}\frac{\Gamma(\frac{n-\beta}{2})}{2\Gamma(\frac{2+\beta}{2})}\|f\|_p,$$
where
$$\widetilde{I}_\beta f(x)=2^{-\beta}\pi^{-\frac{n}{2}}\frac{\Gamma(\frac{n-\beta}{2})}{\Gamma(\frac{\beta}{2})}\int_{\mathbb{R}^n}f(x-y)|y|^{-n+\beta}dy.$$

Moreover,
let $$\gamma(n,\beta)= \left(w_{n-1}+\left(\frac{q \omega_{n-1}}{p' n}\right)^{1/p'}\beta\right) 2^{-\beta} \pi^{-\frac{n}{2}}\frac{\Gamma(\frac{n-\beta}{2})}{2\Gamma(\frac{2+\beta}{2})} .$$
Then, $\gamma(n,\beta)\to \omega_{n-1}\pi^{-\frac{n}{2}}\frac{\Gamma(\frac{n}{2})}{2}$, when $\beta \rightarrow 0^+$.
\end{lemma}	

\begin{lemma}{\rm(\cite{MZ})}\label{lem2.2}
Let $0<p,\,q\leq \infty$ and $T$ be a positive linear operator mapping $L^p$ to $L^q$ with norm $\|T\|_{(p,q)}$. Then $T$ has an $l^r$-valued extension for $1\leq r \leq \infty$,
$$\Big\|\Big(\sum_j|T(f_j)|^r\Big)^{1/r}\Big\|_{q}\leq C \|T\|_{(p,q)}\Big\|\Big(\sum_j|f_j|^r\Big)^{1/r}\Big\|_{p}.$$	
\end{lemma}

\begin{lemma} \label{lem2.3}For any $0<\beta<1/2$,
\begin{equation}\label{eq2.2}
\|\widetilde{\mu}_{\Omega,\beta}f\|_{L^2(\mathbb{R}^n)}\leq C(n)\|\Omega\|_{\infty}\|f\|_{L^{\frac{2n}{n+2\beta}}(\mathbb{R}^n)}.
\end{equation}
\end{lemma}

\begin{proof}
Let $\phi \in C_{0}^{\infty}(\mathbb{R}^n)$ be a radial function such that $0<\phi<1$, $\supp \phi \subset \left\{1/2 \leq |\xi| \leq 2\right\}$ and $\sum_{l \in \mathbb{Z}} \phi^2(2^{-l}\xi)=1$ for $|\xi| \neq 0$. Define the multiplier $\Delta_l$ by $\widehat{\Delta_l f}(\xi)=\phi(2^{-l}\xi)\hat{f}(\xi).$ Then we know
\begin{align*}
\widetilde{\mu}_{\Omega,\beta}f(x)&=  \left ( \sum_{j \in \mathbb{Z}}\int_{1}^{2 } \left | [T_{\beta,j}]_t f(x)\right | ^{2}\frac{dt}{t}   \right )^{1/2 }\\
&= \left ( \sum_{j \in \mathbb{Z}}\int_{1}^{2 } \left |\sum_{l\in \mathbb{Z}} [T_{\beta,j}]_t \Delta_{l-j}^2  f(x)\right | ^{2}\frac{dt}{t}   \right )^{1/2 }\\
&= \left ( \sum_{j \in \mathbb{Z}}\int_{1}^{2 } \left |\sum_{l\in \mathbb{Z}} \Delta_{l-j}[T_{\beta,j}]_t \Delta_{l-j}  f(x)\right | ^{2}\frac{dt}{t}   \right )^{1/2 }	.
\end{align*}
By Minkowski's inequality, we have
\begin{align*}
\|\widetilde{\mu}_{\Omega,\beta}f\|_2 & \leq \Big|\Big|	\left ( \sum_{j \in \mathbb{Z}}\int_{1}^{2 } \left |\sum_{l\in \mathbb{Z}} \Delta_{l-j}[T_{\beta,j}]_t \Delta_{l-j}  f(\cdot)\right | ^{2}\frac{dt}{t}   \right )^{1/2 }\Big|\Big|_2\\
& \leq C \sum_{l \in \mathbb{Z}}
\Big|\Big|\left(\sum_{j \in \mathbb{Z}}\int_{1}^{2}\left|\Delta_{l-j} [T_{\beta,j}]_t \Delta_{l-j}  f(\cdot)\right|^2 \frac{dt}{t}      \right)^{1/2}           \Big|\Big|_2\\
& =:\sum_{l \in \mathbb{Z}} \|V_{\beta,l}f\|_2.
\end{align*}
Recall that
$$[T_{\beta,j}]_t f(x)=  \frac{1}{2^j t}\int_{2^j t \leq \left |x- y \right |\le 2^{j+1}t }  \frac{\Omega (x-y)}{\left | x-y \right |^{n-\beta-1 }  } f(y)dy. $$
Set
$$[K_{\beta,j}]_t (x):= \frac{1}{2^j t}\frac{\Omega(x')}{|x|^{n-\beta-1}}\chi_{\left\{2^j t \leq |x| \leq 2^{j+1}t\right\}},\quad j \in \mathbb{Z} .$$
Then, the operator $[T_{\beta,j}]_t $ can be defined by
$$\widehat{[T_{\beta,j}]_t f}(\xi)=\widehat{[K_{\beta,j}]_t}(\xi)\widehat{f}(\xi).$$

Now we would like to establish that for any $0<\beta<1/2$ and $l \in \mathbb{Z}$,
\begin{equation}\label{eq2.3}
\|V_{\beta,l}f\|_2 \leq C(n)\|\Omega\|_{\infty}\min \{2^{(1+\beta)l},2^{-\frac{1-\beta}{2}l}\}\|f\|_{\frac{2n}{n+2\beta}}.	
\end{equation}

Indeed, by the cancellation of $\Omega$, for any $1\leq t \leq 2$, we have
\begin{equation}\label{eq2.4}
\begin{aligned}
|\widehat{[K_{\beta,j}]_t}(\xi)|&=\Big|\frac{1}{2^j t}\int_{2^j t}^{2^{j+1} t}\int_{\mathbb{S}^{n-1}}\Omega(x') e^{-2\pi i rx'\cdot \xi} d\sigma(x')r^{\beta} dr\Big|	\\
&= \Big|\frac{1}{2^j t}\int_{2^j t}^{2^{j+1} t}\int_{\mathbb{S}^{n-1}}\Omega(x') (e^{-2\pi i rx'\cdot \xi}-1) d\sigma(x')r^{\beta} dr\Big|	\\
& \leq C(n) (2^{j}t)^{\beta} |2^j t \xi| \|\Omega\|_{\infty}\\
& \leq  C(n) \|\Omega\|_{\infty}|\xi|^{-\beta}|2^j \xi|^{1+\beta}.
\end{aligned}
\end{equation}
And
\begin{equation*}
\begin{aligned}
 |\widehat{[K_{\beta,j}]_t}(\xi)|& =\Big|\frac{1}{2^j t}\int_{\mathbb{S}^{n-1}}\Omega(x') \int_{2^j t}^{2^{j+1} t} e^{-2\pi i rx'\cdot \xi} r^{\beta} dr d\sigma(x')\Big|.
\end{aligned}
\end{equation*}
By Van de Corput Lemma, for any $0<\beta_0 <1$, we have
\begin{align*}
 \frac{1}{2^j t}\Big|\int_{2^j t}^{2^{j+1} t} e^{-2\pi i rx'\cdot \xi} r^{\beta} dr\Big| &\leq C (2^{j}t)^{\beta} \min{\left\{|x' \cdot \xi'|^{-1}|2^j t|^{-1},1\right\}}	\\
 &\leq C (2^{j}t)^{\beta}  |x' \cdot \xi'|^{-{\beta}_0}|2^j t|^{-{\beta}_0}.
\end{align*}
Taking $\beta_0 =\frac{1+\beta}{2}$. Then for any $1\leq t \leq 2$, we get
\begin{equation}\label{eq2.5}
\begin{aligned}
|\widehat{[K_{\beta,j}]_t}(\xi)|&\leq C(n)\|\Omega\|_{\infty}2^{j \beta} |2^{j}\xi|^{-{2}/{3}} \int_{\mathbb{S}^{n-1}}|x' \cdot \xi '|^{-{2}/{3}} d\sigma(x')	\\
& \leq C(n) |\xi|^{-\beta}|2^j \xi|^{-\frac{1-\beta}{2}}\|\Omega\|_{\infty}.
\end{aligned}
\end{equation}
By (\ref{eq2.4}) and (\ref{eq2.5}), we obtain
\begin{equation}\label{eq2.6}
	|\widehat{[K_{\beta,j}]_t}(\xi)|\leq C(n)\|\Omega\|_{\infty}|\xi|^{-\beta}\min \left\{|2^j \xi|^{1+\beta},|2^j \xi|^{-\frac{1-\beta}{2}}\right\}.
\end{equation}
Set
$$[m_{\beta,j}]_t(\xi):= \widehat{[K_{\beta,j}]_t}(\xi), \,\,\,\,[m_{\beta,j}^{l}]_t(\xi):= [m_{\beta,j}]_t(\xi)\phi(2^{j-l}\xi).
$$
Define the operator $[T_{\beta,j}^l]_t$ by
$$\widehat{[T_{\beta,j}^l]_t f}(\xi):=\widehat{[T_{\beta,j}]_t \Delta_{l-j}f}(\xi)= [m_{\beta,j}^{l}]_t(\xi) \widehat{f}(\xi).$$
Note that
\begin{equation*}
\supp{[m_{\beta,j}^{l}]_t}(\cdot) \subset \left\{2^{l-1}\leq2^j |\xi| \leq 2^{l+1}\right\}.	
\end{equation*}
This, together with (\ref{eq2.6}), leads to
\begin{equation}\label{eq2.7}
|[m_{\beta,j}^{l}]_t(\xi)|\leq C(n)	\|\Omega\|_{\infty}|\xi|^{-\beta}\min \left\{2^{(1+\beta)l},2^{-\frac{1-\beta}{2}l}\right\}.
\end{equation}
Then by (\ref{eq2.7}) and the Plancherel theorem, for any $1\leq t \leq 2$, we get
\begin{equation}\label{eq2.8}
\begin{aligned}
\|[T_{\beta,j}^l]_t f\|_2  &\leq 	C(n)	\|\Omega\|_{\infty}\min \{2^{(1+\beta)l},2^{-\frac{1-\beta}{2}l}\}\| \widetilde{I}_\beta f\|_2.
\end{aligned}	
\end{equation}
Since $[T_{\beta,j}]_t$ and $\Delta_{l-j}$ are convolution type operators, hence $$[T_{\beta,j}^l]_t=[T_{\beta,j}]_t \Delta_{l-j} = \Delta_{l-j}[T_{\beta,j}]_t. $$
By (\ref{eq2.8}), the Littlewood-Paley theory, taking $q=2,p=\frac{2n}{n+2\beta}$ in Lemma \ref{lem2.1}, replacing $T$ by $\widetilde{I}_\beta $ and let $r=2$ in Lemma \ref{lem2.2}, we have
\begin{equation*}
\begin{aligned}
\|V_{\beta,l}f\|_2&= \Big|\Big|\left(\sum_{j \in \mathbb{Z}}\int_{1}^{2}\left| [T_{\beta,j}^l]_t \Delta_{l-j}  f\right|^2 \frac{dt}{t}      \right)^{1/2}           \Big|\Big|_2 \\
& = \left(\int_{1}^{2}(\sum_{j \in \mathbb{Z}}\Big|\Big|[T_{\beta,j}^l]_t(\Delta_{l-j}f)\Big|\Big|_2^2 \frac{dt}{t}\right)^{1/2}\\
& \leq C(n)\|\Omega\|_{\infty} \min \{2^{(1+\beta)l},2^{-\frac{1-\beta}{2}l}\} \left(\int_{1}^{2}\sum_{j \in \mathbb{Z}}\Big|\Big| \widetilde{I}_\beta(\Delta_{l-j} f)\Big|\Big|_2^2 \frac{dt}{t}\right)^{1/2}\\
& = C(n)\|\Omega\|_{\infty} \min \{2^{(1+\beta)l},2^{-\frac{1-\beta}{2}l}\}\left(\int_{1}^{2}\Big|\Big| \Big(\sum_{j \in \mathbb{Z}} | \widetilde{I}_\beta(|\Delta_{l-j} f|)|^2\Big)^{1/2}\Big|\Big|_2^2 \frac{dt}{t}\right)^{1/2}\\
&\leq C(n)\|\Omega\|_{\infty} \min \{2^{(1+\beta)l},2^{-\frac{1-\beta}{2}l}\}\left(\int_{1}^{2}\Big|\Big| \Big(\sum_{j \in \mathbb{Z}}\left|\Delta_{l-j} f\right|^2 \Big)^{1/2}\Big|\Big|_{\frac{2n}{n+2\beta}}^2 \frac{dt}{t}\right)^{1/2}\\
& \leq C(n)\|\Omega\|_{\infty} \min \{2^{(1+\beta)l},2^{-\frac{1-\beta}{2}l}\} \|f\|_{\frac{2n}{n+2\beta}}.
\end{aligned}	
\end{equation*}
Thus we complete the proof of (\ref{eq2.3}).

Now, for $q=2,\,p=2n/(n+2\beta)$, we have
\begin{equation*}
\begin{aligned}
\|\widetilde{\mu}_{\Omega,\beta} f\|_2  &\leq \sum_{l \in \mathbb{Z}}\|V_{\beta,l}f\|_{2}\\	
& \leq C(n)\|\Omega\|_{\infty}\sum_{l \in \mathbb{Z}} \min \{2^{(1+\beta)l},2^{-\frac{1-\beta}{2}l}\} \|f\|_{\frac{2n}{n+2\beta}}\\
& \leq   C(n)\|\Omega\|_{\infty}\left(\sum_{l \leq 0}2^{ (1+\beta)l}+\sum_{l >0}2^{-\frac{ 1-\beta}{2}l}\right)\|f\|_{\frac{2n}{n+2\beta}}\\
& \leq \frac{C(n)}{1-2^{\frac{\beta-1}{2}}}\|\Omega\|_{\infty}\|f\|_{\frac{2n}{n+2\beta}}\\
& \leq C(n)\|\Omega\|_{\infty}\|f\|_{\frac{2n}{n+2\beta}},
\end{aligned}	
\end{equation*}
where we used that $0<\beta<1/2$ in the last inequality. This completes the proof of Lemma \ref{lem2.3}
\end{proof}

\begin{lemma}\label{lem2.4}
Let $\Omega$ be homogeneous of degree zero and have mean value. Suppose that $\Omega \in L^{\infty}(\mathbb{S}^{n-1})$, $\beta\in (0,\frac{1}{2})$ and $q_0=\frac{n}{n-\beta}$. Then for any $l \in \mathbb{N}$,
\begin{equation}\label{eq2.9}
|\{x: \widetilde{\mu}_{\Omega,\beta}^l f(x)>\lambda\}|^{1/q_0}\leq C(n)\|\Omega\|_{\infty}\frac{l\cdot \|f\|_1}{\lambda}.
\end{equation}
\end{lemma}

\begin{proof} We firs show that for every $l\in\mathbb{Z}$ and any $0<\beta<1/2$,
 \begin{equation}\label{eq2.10}
	\| \widetilde{\mu}_{\Omega,\beta}^l f\|_{L^2(\mathbb{R}^n)}\leq C(n)\|\Omega\|_{\infty}\|f\|_{L^{\frac{2n}{n+2\beta}}}.
\end{equation}

Indeed, for any $0<\beta<1/2$, take $\theta=\frac1{4(1-\beta)}$, by Fourier transform estimates (\ref{eq2.1}), (\ref{eq2.6}), Plancherel's theorem and Lemma \ref{lem2.1}, we have that for every $l\in\mathbb{Z}$,
\begin{equation}\label{2.11}
\begin{aligned}
&\| \widetilde{\mu}_{\Omega,\beta} f-\widetilde{\mu}_{\Omega,\beta}^l f\|_{L^2(\mathbb{R}^n)}^2 \\
&\qquad\leq\int_1^2\Big\|\Big(\sum_{j \in \mathbb{Z}}| \lbrack T_{\beta,j}\rbrack_t f(\cdot)- \lbrack F_{\beta,j}^l\rbrack_t f(\cdot) |^2\Big)^{1/2}\Big\|_{L^2(\mathbb{R}^n)}^2\frac{dt}{t}\\
&\qquad=\int_1^2\sum_{j \in \mathbb{Z}}\int_{\mathbb{R}^n}|\widehat{[K_{\beta,j}]_t}(\xi)|^2| \widehat{\varphi_{j-l}}	(\xi)-1 |^2 |\widehat{f}(\xi)|^2d\xi \frac{dt}{t}\\
&\qquad\leq\|\Omega\|_{\infty}^2\int_{\mathbb{R}^n}\Big(\sum_{j\leq \theta l/2-\ln |\xi|}|2^j\xi|^{2(1+\beta)}|2^{j-l}\xi|^2\\
&\qquad\qquad+\sum_{j >\theta l/2-\ln |\xi|} |2^j \xi |^{-(1-\beta)}\Big)|\xi|^{-2\beta} |\widehat{f}(\xi)|^2d\xi\\
&\qquad\leq C(n) \|\Omega\|_{\infty}^2 \int_{\mathbb{R}^n} \Big(\frac{2^{-2l}{2}^{\theta (\beta+2)l}}{1-2^{-2(2+\beta)}}+\frac{2^{\frac{-\theta(1-\beta)}{2}l}}{1-2^{\beta-1}}\Big) |\xi|^{-2\beta} |\widehat{f}(\xi)|^2d\xi\\
&\qquad=  C(n) \|\Omega\|_{\infty}^2  \Big(\frac{2^{-2l}{2}^{\theta (\beta+2)l}}{1-2^{-2(2+\beta)}}+\frac{2^{\frac{-\theta(1-\beta)}{2}l}}{1-2^{\beta-1}}\Big)\|\widehat{I_{\beta}f}\|_2^2\\
&\qquad\leq C(n)\frac{2^{-{l}/{8}}}{1-2^{\beta-1}} \|\Omega\|_{\infty}^2\|f\|_{L^{\frac{2n}{n+2\beta}}}^2\\
&\qquad\leq C(n) 2^{-{l}/{8}}\|\Omega\|_{\infty}^2\|f\|_{L^{\frac{2n}{n+2\beta}}}^2,
\end{aligned}
\end{equation}
where we used $0<\beta<1/2$ in the last inequality. This, together with Lemma \ref{lem2.3}, implies that (\ref{eq2.10}) holds.

Now we prove (\ref{eq2.9}). Applying Calder\'on-Zygmund decomposition to $f$ at height $\eta=\frac{\lambda^{q_0}}{\|\Omega\|_{\infty}^{q_0}\|f\|_1^{q_0-1}}$, we obtain a disjoint family of dyadic cubes $\{Q_i\}$, such that
$$\sum_{i}|Q_i|\leq \eta^{-1}\|f\|_{L^1}\leq \|\Omega\|_{L^{\infty}}^{q_0}\Big(\frac{\|f\|_1}{\lambda}\Big)^{q_0},$$
which gives $f=g+b$, $\|b\|_{L^{\infty}}\leq 2^n \eta$, $\|b\|_1\leq \|f\|_1$, and
$$b=\sum_i b_i,\supp b_i \subset Q_i,\int_{\mathbb{R}^n}b_i(x)dx=0,\sum_i \|b_i\|\leq \|f\|_1.$$
By Chebychev's inequality and (\ref{eq2.10}), we get
\begin{align*}
|\{x: \widetilde{\mu}_{\Omega,\beta}^l g(x)>\lambda\}|^{1/q_0}
& \leq \Big(\frac{\| \widetilde{\mu}_{\Omega,\beta}^l g\|_{L^2}^2}{\lambda^2}\Big)^{1/q_0}\\
& \leq \Big(C(n)	\frac{\|\Omega\|_{L^{\infty}}^2}{(1-2^{\beta-1})^2}\frac{\|g\|_{L^{\frac{2n}{n+2\beta}}}^2}{\lambda^2}\Big)^{1/q_0}\\
& \leq \Big({C(n)	\frac{\|\Omega\|_{L^{\infty}}^2}{(1-2^{\beta-1})^2}\frac{\|g\|_{L^{\infty}}^{\frac{n-2\beta}{n}} \|g\|_{L^{1}}^{\frac{n+2\beta}{n}}}{\lambda^2}}\Big)^{1/q_0}\\
& \leq  \Big(\frac{C(n)}{(1-2^{\beta-1})^2}\Big)^{\frac{n-\beta}{n}}\|\Omega\|_{L^{\infty}}\frac{\|f\|_1}{\lambda}\\
& \leq  C(n)\|\Omega\|_{L^{\infty}}\frac{\|f\|_1}{\lambda},
\end{align*}
where we use $0<\beta<1/2$ in the last inequality.

Let $E=\cup_i \tilde{Q_i}=\cup_i 4n Q_i$. It is obvious that $|E|\leq C(n) \|\Omega\|_{L^{\infty}}^{q_0}\Big(\frac{\|f\|_1}{\lambda}\Big)^{q_0}. $ The proof of (\ref{eq2.9}) is reduced to prove that
\begin{equation}\label{eq2.11}
	|\{x\in \mathbb{R}^n: \widetilde{\mu}_{\Omega,\beta}^l b(x)>\lambda\}|^{1/q_0}\leq C(n)\|\Omega\|_{\infty}l\frac{\|f\|_1}{\lambda}.
\end{equation}

In what follows, we prove (\ref{eq2.11}). For each fixed cube $Q_i$, let $y_i$ be the center of $Q_i$. For $x,y,x \in \mathbb{R}^n$, set
$$
S_{\beta,t}^{j,l}(x;y,z)=|[K_{\beta,j}]_t \ast \varphi_{j-l}(x-y)-[K_{\beta,j}]_t \ast \varphi_{j-l}(x-z)|.
$$
For $l \in \mathbb{N}$, $x \in 2^{k+2}nQ_i\setminus 2^{k+1}nQ_i$ and $y \in Q_i$,
there are facts that
$$|x-y_i|\approx |x-y| \approx 2^{k+1}nl(Q_i),$$
and
$$\|\varphi_{j-l}(x-y-\cdot)-\varphi_{j-l}(x-y_i-\cdot)\|_{L^1(\mathbb{R}^n)}\leq C(n) \min\{1,2^{l-j}|y-y_i|\}.$$
Thus, by the above facts and $\supp [K_{\beta,j}]_t \ast \varphi_{j-l} \subset \{x \in \mathbb{R}^n:2^{j-2}\leq|x|\leq 2^{j+2}\}$, we deduce that for any $x \in 2^{k+2}nQ_i\setminus 2^{k+1}nQ_i $ and $0<\beta<1/2$,
\begin{equation}\label{eq2.12}
\begin{aligned}
&\sum_{j \in \mathbb{Z}} \sup _{t \in [1,2]}S_{\beta,t}^{j,l}(x;y,y_i) \chi_{\{2^{k+2}nQ_i\setminus 2^{k+1}nQ_i\}}(x)\\
=&  \sum_{j \in \mathbb{Z}} \sup _{t \in [1,2]} |[K_{\beta,j}]_t \ast \varphi_{j-l}(x-y)-[K_{\beta,j}]_t \ast \varphi_{j-l}(x-y_i)| \\
\leq &  \sum_{j \in \mathbb{Z}}\int_{\mathbb{R}^n} \sup _{t \in [1,2]} |[K_{\beta,j}]_t(z)\|\varphi_{j-l}(x-y-z)-\varphi_{j-l}(x-y_i-z) |dz\\
\leq & C(n)\|\Omega\|_{\infty}\sum_{j \in \mathbb{Z}: 2^j\approx 2^{k+1}nl(Q_i)}\frac{1}{(2^{j-2})^{n-\beta}}	\min\{1,2^{l-j}|y-y_i|\}\\
\leq & C(n) \|\Omega\|_{\infty}\sum_{j \in \mathbb{Z}: 2^j\approx |x-y_i|}\frac{1}{(2^{j-2})^{n-\beta}}	\min\{1,2^{l-j}|y-y_i|\}\\
\leq & C(n) \|\Omega\|_{\infty} \frac{1}{|x-y_i|^{n-\beta}}\min \{1,2^l \frac{|y-y_i|}{|x-y_i|}\}\\
=& : C(n) \|\Omega\|_{\infty} \frac{1}{|x-y_i|^{n-\beta}}\omega_l(\frac{|y-y_i|}{|x-y_i|}),
\end{aligned}	
\end{equation}
where $\omega_l(t)=\min \{1,2^lt\}$
and for $l \in \mathbb{N}$,
\begin{equation}\label{eq2.13}
\|\omega_l\|_{Dini}=\int_{0}^{1}w_l(d)\frac{dt}{t}\leq 1+l\leq 2l.	
\end{equation}
Here we give the details of the third to last inequality:
\begin{align*}
&\sum_{j:2^j \approx 2^{k+1}nl(Q_i)} \frac{1}{2^{j(n-\beta)}} \min\{1,2^{l-j}|y-y_i|\}\\
&\qquad\qquad= \sum_{2^{k-3}nl(Q_i)\leq2^j \leq 2^{k+4}n\sqrt{n}l(Q_i)} \frac{1}{2^{j(n-\beta)}} \min\{1,2^{l-j}|y-y_i|\}\\
&\qquad\qquad\leq\sum_{ \frac{1}{2^4 \sqrt{n}}\leq |x-y_i| \leq 2^j \leq 2^4 \sqrt{n}|x-y_i|} \frac{1}{2^{j(n-\beta)}} \min\{1,2^{l-j}|y-y_i|\}\\
&\qquad\qquad\leq\sum_{ \frac{1}{2^4 \sqrt{n}}\leq |x-y_i| \leq 2^j \leq 2^4 \sqrt{n}|x-y_i|} \frac{1}{2^{j(n-\beta)}} \min\{1,2^{l+4}\sqrt{n}\frac{|y-y_i|}{|x-y_i|}\}\\
&\qquad\qquad\leq\frac{(2^4 \sqrt{n})^{n-\beta}-(2^4 \sqrt{n})^{-(n-\beta)}}{1-2^{-(n-\beta)}}\min\{1,2^{l}\frac{|y-y_i|}{|x-y_i|}\} \frac{1}{|x-y_i|^{n-\beta}}\\
&\qquad\qquad\leq C(n) \min\{1,2^{l}\frac{|y-y_i|}{|x-y_i|}\} \frac{1}{|x-y_i|^{n-\beta}},
\end{align*}
where we use $0<\beta<1/2$ in the last inequality.

According to (\ref{eq2.12}) and (\ref{eq2.13}), we get
\begin{equation}\label{eq2.14}
\begin{aligned}
	&\Big(\sum_{k=1}^{+\infty}\int_{2^{k+2}nQ_i\setminus 2^{k+1}nQ_i}|\sum_{j \in \mathbb{Z}} \sup _{t \in [1,2]}S_{\beta,t}^{j,l}(x;y,y_i) |^{q_0} dx\Big)^{1/q_0}\\
	&\qquad\qquad\leq  \Big(\sum_{k=1}^{+\infty}\int_{2^{k+2}nQ_i\setminus 2^{k+1}nQ_i}| \frac{1}{|x-y_i|^{n-\beta}}w_l(\frac{|y-y_i|}{|x-y_i|}) |^{q_0} dx\Big)^{1/q_0} \\
	&\qquad\qquad\leq C(n)\|\Omega\|_{\infty}\sum_{k=1}^{+\infty}w_l(2^{-k})  \Big(\int_{2^{k+2}nQ_i\setminus 2^{k+1}nQ_i} \frac{1}{|x-y_i|^{n}}dx\Big)^{1/q_0} \\
	&\qquad\qquad\leq C(n)\|\Omega\|_{\infty}\|\omega_l\|_{\rm Dini}\le C(n)\|\Omega\|_{\infty}l.
\end{aligned}	
\end{equation}
Thus, applying Chebychev's inequality, a trivial computation involving Minkowski's inequality, vanishing moment of $b_i$ and by (\ref{eq2.14}), we get
\begin{equation*}
\begin{split}
	&|\{x\in \mathbb{R}^n\setminus E: \widetilde{\mu}_{\Omega,\beta}^l b(x)>\lambda\}|^{1/q_0}\\
	&\qquad\le\lambda^{-1}\sum_{i}\Big(\int_{x \notin \tilde{Q_i}}|\widetilde{\mu}_{\Omega,\beta}^l b_i(x)|^{q_0}dx\Big)^{1/q_0}\\
	&\qquad\le \lambda^{-1}\sum_{i}\Big(\int_{x \notin \tilde{Q_i}}
		\Big(\int_1^2  \sum_{j \in \mathbb{Z}} \Big(\int_{\mathbb{R}^n}S_{\beta,t}^{j,l}(x;y,y_i)|b_i(y)|dy\Big)^2\frac{dt}{t}\Big)^{q_0/2}
	dx\Big)^{1/q_0}\\
	&\qquad\le \lambda^{-1}\sum_{i}\Big(\int_{x \notin \tilde{Q_i}}
		\Big(  \sum_{j \in \mathbb{Z}} \int_{\mathbb{R}^n}\Big(\int_1^2 \{S_{\beta,t}^{j,l}(x;y,y_i)\}^2 \frac{dt}{t}\Big)^{1/2}|b_i(y)|dy\Big)^{q_0}
	dx\Big)^{1/q_0}\\
		&\qquad\le \lambda^{-1}\sum_{i}\Big(\int_{x \notin \tilde{Q_i}}
		\Big(  \sum_{j \in \mathbb{Z}} \int_{\mathbb{R}^n}\sup _{t \in [1,2]}S_{\beta,t}^{j,l}(x;y,y_i) |b_i(y)|dy\Big)^{q_0}
	dx\Big)^{1/q_0}\\
	&\qquad\le \lambda^{-1}\sum_i \int_{Q_i}|b_i(y)|\Big(\int_{x \notin \tilde{Q_i}}|\sum_{j \in \mathbb{Z}} \sup _{t \in [1,2]}S_{\beta,t}^{j,l}(x;y,y_i) |^{q_0} dx\Big)^{1/q_0}dy\\
	&\qquad=   \lambda^{-1}\sum_i \int_{Q_i}|b_i(y)|\Big(\sum_{k=1}^{+\infty}\int_{2^{k+2}nQ_i\setminus 2^{k+1}nQ_i}|\sum_{j \in \mathbb{Z}} \sup _{t \in [1,2]}S_{\beta,t}^{j,l}(x;y,y_i) |^{q_0} dx\Big)^{1/q_0}dy\\
	&\qquad\le\lambda^{-1} C(n)\|\Omega\|_{\infty}l \sum_i \|b_i\|_{L^1(\mathbb{R}^n)}\\
	&\qquad\le C(n)\|\Omega\|_{\infty}l \frac{\|f\|_{L^1(\mathbb{R}^n)}}{\lambda}.
\end{split}
\end{equation*}
This completes the proof of (\ref{eq2.11}) and Lemma \ref{lem2.4} is proved.
\end{proof}

\begin{lemma}\label{lem2.5} Let $\Omega$ be homogeneous of degree zero and have mean value. Suppose that $\Omega \in L^{\infty}(\mathbb{S}^{n-1})$. Let $\beta\in (0,\frac{1}{2})$ and $q_0=\frac{n}{n-\beta}$. Then for any $l \in \mathbb{N}$,
\begin{equation}\label{eq2.15}
|\{x: \mathcal{M}_{\widetilde{\mu}_{\Omega,\beta}^l}f(x)>\lambda\}|^{1/q_0}\leq C(n)\|\Omega\|_{\infty}\frac{l \cdot \|f\|_1}{\lambda}.
\end{equation}
\end{lemma}

\begin{proof} Let $x \in \mathbb{R}^n$ and $Q \subset \mathbb{R}^n$ be a cube containing $x$. Define $B_x=B(x,2\sqrt{n}l(Q))$. Then $3Q \subset B_x$. For each $\xi \in Q$, we split
\begin{align*}
| \widetilde{\mu}_{\Omega,\beta}^l (f\chi_{\mathbb{R}^n\setminus 3Q})(\xi)|& \leq | \widetilde{\mu}_{\Omega,\beta}^l (f\chi_{\mathbb{R}^n\setminus B_x})(\xi) - \widetilde{\mu}_{\Omega,\beta}^l (f\chi_{\mathbb{R}^n\setminus B_x})(x)|\\
 	&  \quad + | \widetilde{\mu}_{\Omega,\beta}^l (f\chi_{B_x \setminus 3Q})(\xi)|+ | \widetilde{\mu}_{\Omega,\beta}^l (f\chi_{\mathbb{R}^n\setminus B_x})(x)|\\
 	&=: \uppercase\expandafter{\romannumeral1}+ \uppercase\expandafter{\romannumeral2}+ \uppercase\expandafter{\romannumeral3}.
\end{align*}
To estimate $\uppercase\expandafter{\romannumeral1} $, we set
$$\widetilde{S}_{\beta,t}^{j,l}(x;y,\xi):=|[K_{\beta,j}]_t \ast \varphi_{j-l}(x-y)-[K_{\beta,j}]_t \ast \varphi_{j-l}(\xi -y)|.
$$
For $l \in \mathbb{N}$, $y \in 2^{k}B_x\setminus 2^{k-1}B_x$ and $x,\xi \in Q$,
there are facts that
$$2^k \sqrt{n}l(Q)\leq|x-y|\leq 2^{k+1} \sqrt{n}l(Q),|x-\xi|\leq \sqrt{n}l(Q) ,$$
and
$$\|\varphi_{j-l}(x-y-\cdot)-\varphi_{j-l}(\xi-y-\cdot)\|_{L^1(\mathbb{R}^n)}\leq C(n) \min\{1,2^{l-j}|x-\xi|\}.$$
Thus, by the above facts and $\supp [K_{\beta,j}]_t \ast \varphi_{j-l} \subset \{x \in \mathbb{R}^n:2^{j-2}\leq|x|\leq 2^{j+2}\}$, we can get the similarity estimate as (\ref{eq2.12}) for each $x,\xi \in Q$ and $0<\beta<1/2$,
\begin{equation*}
\begin{aligned}
&\sum_{j \in \mathbb{Z}} \sup _{t \in [1,2]}\widetilde{S}_{\beta,t}^{j,l}(x;y,\xi) \chi_{\{2^{k+2}nQ_i\setminus 2^{k+1}nQ_i\}}(y)\\
=&  \sum_{j \in \mathbb{Z}} \sup _{t \in [1,2]} |[K_{\beta,j}]_t \ast \varphi_{j-l}(x-y)-[K_{\beta,j}]_t \ast \varphi_{j-l}(\xi-y)| \\
\leq &  \sum_{j \in \mathbb{Z}}\int_{\mathbb{R}^n} \sup _{t \in [1,2]} |[K_{\beta,j}]_t(z)\|\varphi_{j-l}(x-y-z)-\varphi_{j-l}(\xi-y-z) |dz\\
\leq & C(n)\|\Omega\|_{\infty}\sum_{j \in \mathbb{Z}: 2^j\approx 2^{k}\sqrt{n}l(Q)}\frac{1}{(2^{j-2})^{n-\beta}}	\min\{1,2^{l-j}|x-\xi|\}\\
= & C(n) \|\Omega\|_{\infty}\sum_{j \in \mathbb{Z}: 2^j\approx |x-y|}\frac{1}{(2^{j-2})^{n-\beta}}	\min\{1,2^{l-j}|x-\xi|\}\\
\leq & C(n) \|\Omega\|_{\infty} \frac{1}{|x-y|^{n-\beta}}\min \{1,2^l \frac{|x-\xi|}{|x-y|}\}\\
=& C(n) \|\Omega\|_{\infty} \frac{1}{|x-y|^{n-\beta}}w_l(\frac{|x-\xi|}{|x-y|}).
\end{aligned}	
\end{equation*}
where we use $0<\beta<1/2$ in the second to last inequality. From this, we have
\begin{equation}\label{eq2.16}
\begin{aligned}
\uppercase\expandafter{\romannumeral1}
& \leq \Big(\int_1^2 \sum_{j \in \mathbb{Z}}\Big| \int_{\mathbb{R}^n\setminus B_x} \widetilde{S}_{\beta,t}^{j,l}(x;y,\xi)|f(y)|dy\Big| ^2\frac{dt}{t}\Big)^{1/2}\\
& \leq \Big(\int_1^2 \Big| \sum_{j \in \mathbb{Z}}\int_{\mathbb{R}^n\setminus B_x} \widetilde{S}_{\beta,t}^{j,l}(x;y,\xi)|f(y)|dy\Big| ^2\frac{dt}{t}\Big)^{1/2}\\
&=  \Big(\int_1^2 \Big| \sum_{k=1}^{\infty}\int_{2^{k}B_x\setminus 2^{k-1}B_x} \sum_{j \in \mathbb{Z}}\widetilde{S}_{\beta,t}^{j,l}(x;y,\xi)|f(y)|dy\Big| ^2\frac{dt}{t}\Big)^{1/2}\\
&\leq C(n)\|\Omega\|_{\infty}  \Big(\int_1^2 \Big| \sum_{k=1}^{\infty}\int_{2^{k}B_x\setminus 2^{k-1}B_x} w_l(\frac{|x-\xi|}{|x-y|})\frac{|f(y)|}{|x-y|^{n-\beta}} dy\Big| ^2\frac{dt}{t}\Big)^{1/2}\\
&\leq C(n)\|\Omega\|_{\infty}  \Big(\int_1^2 \Big| \sum_{k=1}^{\infty}w_l(2^{-k})\int_{2^{k}B_x\setminus 2^{k-1}B_x} \frac{|f(y)|}{|x-y|^{n-\beta}} dy\Big| ^2\frac{dt}{t}\Big)^{1/2}\\
&\leq C(n)\|\Omega\|_{\infty} \|w_l\|_{Dini}M_{\beta}f(x)\\
&\leq C(n)l\|\Omega\|_{\infty}  M_{\beta}f(x),
\end{aligned}	
\end{equation}
where
$$M_{\beta}f(x)=\sup_{r>0} \frac{1}{r^{n-\beta}}\int_{|x-y|\leq r}|f(y)|dy.$$
It is obvious that for any $x \in \mathbb{R}^n$ and $l \in \mathbb{N}$,
\begin{equation}\label{eq2.17}
	\sup_{t \in [1,2]}|[K_{\beta,j}]_t \ast \varphi_{j-l}(x)|\leq C(n) \|\Omega\|_{\infty}{|x|}^{-(n-\beta)}\chi_{\{2^{j-2}\leq |x|\leq 2^{j+2}\}}(x).
\end{equation}

For $\uppercase\expandafter{\romannumeral 2} $, observe that $x,\xi \in Q$ and $y \in B_x \setminus 3Q$,
$$l(Q)\leq|x-y|\leq 2\sqrt{n}l(Q),\quad |x-\xi|\leq \sqrt{n}l(Q) ,$$
then
$$l(Q)\leq |\xi -y|\leq 3\sqrt{n}l(Q) .$$
By  $\supp [K_{\beta,j}]_t \ast \varphi_{j-l} \subset \{2^{j-2}\leq |\xi-y|\leq 2^{j+2}\}$, for each fixed $t \in [1,2]$ and $j \in \mathbb{Z}$ with $2^j \approx \sqrt{n}l(Q)$. It follows that for any $0<\beta<1/2$,
\begin{equation}\label{eq2.18}
	\begin{aligned}
		\uppercase\expandafter{\romannumeral 2}
		 & = \Big(\int_1^2 \sum_{j \in \mathbb{Z}}\vert \int_{B_x\setminus 3Q} |[K_{\beta,j}]_t \ast \varphi_{j-l}(\xi-y)f(y)| dy\vert ^2\frac{dt}{t}\Big)^{1/2}\\
		 &  \leq C(n) \|\Omega\|_{\infty}  \Big(\int_1^2  \sum_{j: 2^j \approx \sqrt{n}l(Q)}\vert \frac{1}{2^{j(n-\beta)}} \int_{B_x\setminus 3Q} |f(y)| dy\vert ^2\frac{dt}{t}\Big)^{1/2}\\
		&  \leq C(n) \|\Omega\|_{\infty}  \Big(\int_1^2  \vert \sum_{j: 2^j \approx \sqrt{n}l(Q)}\frac{1}{2^{j(n-\beta)}} \int_{B_x\setminus 3Q} |f(y)| dy\vert ^2\frac{dt}{t}\Big)^{1/2}\\
		&  \leq C(n) \|\Omega\|_{\infty}  \sum_{j: 2^j \approx \sqrt{n}l(Q)} \frac{1}{2^{j(n-\beta)}} \int_{B_x\setminus 3Q} |f(y)| dy\\
	&\leq C(n)\|\Omega\|_{\infty}  M_{\beta}f(x)	 		 		 		 		 		,
	\end{aligned}
\end{equation}
where we use $0<\beta<1/2$ in the last inequality.

For $\uppercase\expandafter{\romannumeral3} $, write
 \begin{equation}\label{eq2.19}
 \begin{aligned}
 	\uppercase\expandafter{\romannumeral3}
 	& \leq \widetilde{\mu}_{\Omega,\beta}^l f(x)+\Big(\int_1^2 \sum_{j \in \mathbb{Z}}|[F_{\beta,j}^l]_t(f \chi_{B_x})(x)|^2\frac{dt}{t}\Big)^{1/2}\\
 	& = \widetilde{\mu}_{\Omega,\beta}^l f(x)+\Big(\int_1^2 \sum_{j: 2^j \leq 8\sqrt{n}l(Q)}|[F_{\beta,j}^l]_t(f \chi_{B_x})(x)|^2\frac{dt}{t}\Big)^{1/2}\\
 	 	& \leq 2\widetilde{\mu}_{\Omega,\beta}^l f(x)+\Big(\int_1^2 \sum_{j: 2^j \leq 8\sqrt{n}l(Q)}|[F_{\beta,j}^l]_t(f \chi_{\mathbb{R}^n \setminus B_x})(x)|^2\frac{dt}{t}\Big)^{1/2}\\
 	 	& =: 2\widetilde{\mu}_{\Omega,\beta}^l f(x)+Df(x). 	
 \end{aligned}	
 \end{equation}

Based on (\ref{eq2.17}), we obtain
\begin{equation}\label{eq2.20}
\begin{aligned}
Df(x)& \leq   \sum_{j: 2^j \leq 8\sqrt{n}l(Q)} \int_{\mathbb{R}^n\setminus B_x} \sup_{t \in [1,2]}|[K_{\beta,j}]_t \ast \varphi_{j-l}(x-y)f(y)| dy\\	& = \sum_{j: \frac{\sqrt{n}l(Q)}{2}\leq 2^j \leq 8\sqrt{n}l(Q)} \int_{\mathbb{R}^n\setminus B_x} \sup_{t \in [1,2]}|[K_{\beta,j}]_t \ast \varphi_{j-l}(x-y)f(y)| dy\\
& \leq C(n)\|\Omega\|_{\infty} \sum_{j: \frac{\sqrt{n}l(Q)}{2}\leq 2^j \leq 8\sqrt{n}l(Q)} 2^{-j(n-\beta)}\int_{2^{j-2}\leq |x-y|\leq 2^{j+2}}|f(y)|dy\\
& \leq C(n)\|\Omega\|_{\infty}   \sum_{j: \frac{\sqrt{n}l(Q)}{2}\leq 2^j \leq 8\sqrt{n}l(Q)} 2^{-j(n-\beta)}\int_{ |x-y|\leq 16 \sqrt{n}l(Q)}|f(y)|dy\\
& \leq C(n)\|\Omega\|_{\infty} M_{\beta}f(x),
\end{aligned}	
\end{equation}
where we use $0<\beta<1/2$ in the last inequality.

Combining with (\ref{eq2.16}) and (\ref{eq2.18})-(\ref{eq2.20}) yields that
$$\mathcal{M}_{\widetilde{\mu}_{\Omega,\beta}^l}f(x) \leq C(n)(\|\Omega\|_{\infty}l M_{\beta}f(x)+ \widetilde{\mu}_{\Omega,\beta}^l f(x)).$$
 Thus, by the weak type $(L^1,L^{q_0,\infty})$ of the fractional maximal operator $M_{\beta}$ (see \cite{CG}) and Lemma \ref{lem2.4}, we complete the proof of (\ref{eq2.15}). Lemma \ref{lem2.5} is proved.
\end{proof}

\section{Sparse domination}
In this section, we will establish the sparse domination of $\widetilde{\mu}_{\Omega,\beta}^l $, which is the key to obtain our main theorems.

\begin{lemma}\label{lem3.1}
Let $\beta \in (0,1/2)$ and $\Omega \in L^{\infty}(\mathbb{S}^{n-1})$. For every compactly supported $f \in L^1(\mathbb{R}^n)$, there exists a $\frac{1}{2\cdot 3^n}$-sparse family $\mathcal{S}$ such that for almost every $x \in \mathbb{R}^n$,
\begin{equation}\label{eq3.1}
\widetilde{\mu}_{\Omega,\beta}^l f(x) \leq C(n)\|\Omega\|_{\infty}l\cdot  \mathcal{A}_{\mathcal{S}}^{2,\beta} f(x)	.
\end{equation}   	
\end{lemma}

\begin{proof}
For a fixed cube $Q_0 \subset \mathbb{R}^n$, let us first show that there exists a $\frac{1}{2}$-sparse family $\mathcal{F}\subset \mathcal{D}(Q_0)$ such that for almost every $x \in Q_0$,
\begin{equation}\label{eq3.2}
\Big(\widetilde{\mu}_{\Omega,\beta}^l (f\chi_{3Q_0})(x) \Big)^2\leq C(n)\|\Omega\|_{\infty}^2 l^2\cdot \sum_{Q \in \mathcal{F}}|Q|^{\frac{2\lambda}{n}}\langle |f|\rangle_{3Q}^2\chi_{Q}(x) .
\end{equation}

To prove (\ref{eq3.2}), it suffices to prove the following recursive estimate: there exist pairwise disjoint cubes $P_j \in \mathcal{D}(Q_0) $ such that $\sum_j |P_j|\leq \frac{1}{2}|Q_0|$ and for almost every $x \in Q_0$,
\begin{equation}\label{eq3.3}
\begin{aligned}
	\Big(\widetilde{\mu}_{\Omega,\beta}^l  (f\chi_{3Q_0})(x) \Big)^2
	\leq & C(n)\|\Omega\|_{\infty}^2 l^2 \cdot |Q|^{\frac{2\lambda}{n}}\langle |f|\rangle_{3Q}^2\\
	& +\sum_j  \Big(\widetilde{\mu}_{\Omega,\beta}^l (f\chi_{3Q_0})(x) \Big)^2\chi_{P_j}(x).
	\end{aligned}
\end{equation}
Then iterating this estimate we obtain (\ref{eq3.2}) with $\mathcal{F}=\{P_j^k\},k \in \mathbb{Z}$, where $\{P_j^0\}=\{Q_0\}$, $\{P_j^1\}= \{P_j\} $
and $\{P_j^k\} $ are the cubes obtained at the $k$-th stage of the iterative process.

Indeed, for each $\{P_j^k\} $ it suffices to choose
$$E_{P_j^k}= P_j^k\setminus \cup_j P_j^{k+1} .$$
Then $\mathcal{F}$ is a $\frac{1}{2}$-sparse family. In fact, for any $P_j^k \subseteq \mathcal{F}$, we have
$$| E_{P_j^k} |=| P_j^k |-\sum_j | P_j^{k+1} |\ge | P_j^k | -\frac{1}{2} | P_j^k | =\frac{1}{2} | P_j^k | .$$

Next we prove (\ref{eq3.3}). Given a cube $Q_0$, for $x \in Q_0$ define a local version of $\mathcal{M}_{\widetilde{\mu}_{\Omega,\beta}^l }$ by
 $$\mathcal{M}_{\widetilde{\mu}_{\Omega,\beta}^l, Q_0}f(x)=\sup_{Q \ni x, Q \subset Q_0}\Big\| \left ( \int_{1}^{2 } \sum_{j =J_Q}^{\infty}\left | \lbrack F_{\beta,j}^l\rbrack_t f(\cdot)\right | ^{2}\frac{dt}{t}   \right )^{1/2 }\Big\|_{L^{\infty}(Q)},$$
where and in follows, for a cube $Q \subset \mathbb{R}^n$, $J_Q ,J_{Q}^* $ are the integers such that $2^{J_Q-1}\leq 4l(Q)< 2^{J_Q}$ and $2^{J_Q^* -1}\leq 16nl(Q)< 2^{J_Q^*} $. Let $x \in \mathbb{R}^n$, $Q \subset Q_0$ such that $x \in Q$. For each $\xi \in Q$, write
\begin{align*}
\left ( \int_{1}^{2 } \sum_{j =J_Q}^{\infty}\left | \lbrack F_{\beta,j}^l\rbrack_t (f\chi_{3Q_0})(\xi)\right | ^{2}\frac{dt}{t}   \right )^{1/2 }
= & \left ( \int_{1}^{2 } \sum_{j =J_Q}^{J_Q^*}\left | \lbrack F_{\beta,j}^l\rbrack_t (f\chi_{3Q_0})(\xi)\right | ^{2}\frac{dt}{t}   \right )^{1/2 }\\	& + \left ( \int_{1}^{2 } \sum_{j =J_Q^*}^{\infty}\left | \lbrack F_{\beta,j}^l\rbrack_t (f\chi_{3Q_0})(\xi)\right | ^{2}\frac{dt}{t}   \right )^{1/2 }	\\
=& :D_1f(\xi)+D_2f(\xi).
\end{align*}
Applying (\ref{eq2.17}), for any $0<\beta<1/2$,
\begin{align*}
D_1 f(\xi) & \leq C(n)\|\Omega\|_{\infty}  \Big(\sum_{j =J_Q}^{J_Q^*}\frac{1}{2^{2j(n-\beta)}}\Big(\int_{2^{j-2}\leq |x-y|\leq 2^{j+2}}|f(y)| \chi_{3Q_0}(y) dy \Big)^2\Big)^{1/2}\\
& \leq C(n)\|\Omega\|_{\infty}  \Big(\sum_{j =J_Q}^{J_Q^*}\frac{1}{2^{2j(n-\beta)}}\Big(\int_{ |x-y|\leq 32nl(Q)}|f(y)| \chi_{3Q_0}(y) dy \Big)^2\Big)^{1/2}\\
& \leq 	C(n)\|\Omega\|_{\infty}M_{\beta}(f\chi_{3Q_0})(x),
\end{align*}
where we use $0<\beta<1/2$ in the last inequality.

Note that for each $t \in [1,2]$ and $j \ge J_Q^*$,
\begin{align*}
\lbrack F_{\beta,j}^l\rbrack_t (f\chi_{3Q_0})(\xi) &= \lbrack F_{\beta,j}^l\rbrack_t (f\chi_{3Q_0\setminus 3Q})(\xi) \\
& = \lbrack F_{\beta,j}^l\rbrack_t (f\chi_{3Q_0}\chi_{\mathbb{R}^n\setminus 3Q})(\xi).
\end{align*}
Then $$D_2f(\xi)\leq \mathcal{M}_{\widetilde{\mu}_{\Omega,\beta}^l}(f \chi_{3Q_0})(x).$$
Therefore,
\begin{equation}\label{eq3.4}
\mathcal{M}_{\widetilde{\mu}_{\Omega,\beta}^l, Q_0}f(x) \leq C(n)\|\Omega\|_{\infty}M_{\beta}(f\chi_{3Q_0})(x) + \mathcal{M}_{\widetilde{\mu}_{\Omega,\beta}^l}(f \chi_{3Q_0})(x).
\end{equation}
Let
\begin{align*}
E=& \{x \in Q_0: \widetilde{\mu}_{\Omega,\beta}^l (f \chi_{3Q_0})(x)>D\|\Omega\|_{\infty} l\cdot |Q_0|^{\frac{\beta}{n}}\langle|f|\rangle_{3Q_0} \}\\
&\cup 	\{x \in Q_0: \mathcal{M}_{\widetilde{\mu}_{\Omega,\beta}^l, Q_0}f(x) >D \|\Omega\|_{\infty} l\cdot |Q_0|^{\frac{\beta}{n}}\langle|f|\rangle_{3Q_0} \},
\end{align*}
where $D$ is a positive constant only depending on $n$. By Lemma \ref{lem2.4} and \ref{lem2.5}, together with (\ref{eq3.4}) and the weak type $(L^1,L^{q_0,\infty})$ of the fractional maximal operator $M_{\beta}$ (see \cite{CG}), we obtain
\begin{align*}
|E|& \leq 2 \Big(C_1(n)l\frac{|Q_0|\langle|f|\rangle_{3Q_0}}{Dl|Q_0|^{\frac{\beta}{n}}\langle|f|\rangle_{3Q_0} }\Big)^{\frac{n}{n-\beta}}+ \Big(C_2(n)\frac{|Q_0|\langle|f|\rangle_{3Q_0}}{Dl|Q_0|^{\frac{\beta}{n}}\langle|f|\rangle_{3Q_0} }\Big)^{\frac{n}{n-\beta}}\\
& \leq 2 [(\frac{C_1(n)}{D})^{\frac{n}{n-\beta}}+ (\frac{C_2(n)}{D})^{\frac{n}{n-\beta}}] |Q_0|.
\end{align*}
Therefore, for $0<\beta<1/2$, choosing $D=D(n)$ large enough, we have that
$$|E|\leq \frac{1}{2^{n+2}}|Q_0|.$$
Let us apply Calder$\mathrm{\acute{o}}$n-Zygmund decomposition to the function $\chi_{E}$ on $Q_0$ at height $\eta=\frac{1}{2^{n+1}} $, then we can obtain pairwise disjoint cubes $P_j \in \mathcal{D}(Q_0)$ such that
$$\sum_j |P_j|\leq \frac{1}{2}|Q_0|,$$
and $$\frac{1}{2^{n+1}}|P_j|\leq |P_j \cap E|\leq \frac{1}{2}|P_j|, $$
so that $|P_j \cap E^c|>0$. And the family also satisfies that
$|E \setminus \cup_j P_j|=0$.
Write
\begin{equation}\label{eq3.5}
\begin{split}
\widetilde{\mu}_{\Omega,\beta}^l (f\chi_{3Q_0})(x)^2\chi_{Q_0}(x)= &\widetilde{\mu}_{\Omega,\beta}^l (f\chi_{3Q_0})(x)^2\chi_{Q_0\setminus \cup_j P_j}(x)\\	
&+\sum_j  \int_{1}^{2 } \sum_{m=J_{P_j}}^{\infty}\left | \lbrack F_{\beta,m}^l\rbrack_t (f\chi_{3Q_0})(x)\right | ^{2}\frac{dt}{t}\chi_{P_j}(x)  \\
&+\sum_j  \int_{1}^{2 } \sum_{m=-\infty}^ {J_{P_j}-1}\left | \lbrack F_{\beta,m}^l\rbrack_t (f\chi_{3Q_0})(x)\right | ^{2}\frac{dt}{t}\chi_{P_j}(x) .
\end{split}
\end{equation}
The facts that $|E \setminus \cup_j P_j|=0$, $\chi_{Q_0\setminus \cup_j P_j}(x) = \chi_{E\setminus \cup_j P_j}(x)+ \chi_{(Q_0\setminus E)\setminus \cup_j P_j}(x) $ and the definition of the set $E$, imply that for almost every $x \in Q_0\setminus \cup_j P_j $,
\begin{equation}\label{eq3.6}
\widetilde{\mu}_{\Omega,\beta}^l (f \chi_{3Q_0})(x)^2 \chi_{Q_0\setminus \cup_j P_j}(x)\leq D(n)\|\Omega\|_{\infty}^2 l^2 \cdot |Q_0|^{\frac{2\beta}{n}}\langle|f|\rangle_{3Q_0}^2\chi_{Q_0}(x).
\end{equation}

By the definition of $\mathcal{M}_{\widetilde{\mu}_{\Omega,\beta}^l, Q_0}$ and the fact that $|P_j \cap E^c|>0$, we deduce that
\begin{equation}\label{eq3.7}
\begin{split}
 &\sum_j  \int_{1}^{2 } \sum_{m=J_{P_j}}^{\infty}\left | \lbrack F_{\beta,m}^l\rbrack_t (f\chi_{3Q_0})(x)\right | ^{2}\frac{dt}{t}\chi_{P_j}(x) \\
 &\qquad\qquad\qquad\leq  \sum_j \inf_{y \in P_j} \mathcal{M}_{\widetilde{\mu}_{\Omega,\beta}^l, Q_0}f(y)^2 \chi_{P_j}(x)\\
 &\qquad\qquad\qquad\leq \sum_j \inf_{y \in P_j \cap E^c} \mathcal{M}_{\widetilde{\mu}_{\Omega,\beta}^l, Q_0}f(y)^2 \chi_{P_j}(x)\\
 &\qquad\qquad\qquad\leq  D(n)\|\Omega\|_{\infty}^2 l^2\cdot  |Q_0|^{\frac{2\beta}{n}}\langle|f|\rangle_{3Q_0}^2\chi_{Q_0}(x).
\end{split}
\end{equation}

On the other hand, it is easy to verify that when $t \in [1,2]$, $x \in P_j$ and $m \leq J_{P_j}-1$,
$$\lbrack F_{\beta,m}^l\rbrack_t (f\chi_{3Q_0\setminus 3P_j})(x)=0 ,$$
and
\begin{align*}
&  \sum_j \int_{1}^{2 } \sum_{m=-\infty}^ {J_{P_j}-1}\left | \lbrack F_{\beta,m}^l\rbrack_t (f\chi_{3Q_0})(x)\right | ^{2}\frac{dt}{t}\chi_{P_j}(x)\\
&\qquad\qquad\qquad\leq \sum_j  \int_{1}^{2 } \sum_{m=-\infty}^ {J_{P_j}-1}\left | \lbrack F_{\beta,m}^l\rbrack_t (f\chi_{3P_j})(x)\right | ^{2}\frac{dt}{t}\chi_{P_j}(x)\\
&\qquad\qquad\qquad\leq  \sum_j \widetilde{\mu}_{\Omega,\beta}^l (f \chi_{3P_j})(x)^2 \chi_{P_j}(x).
\end{align*}
Thus, together with (\ref{eq3.5})-(\ref{eq3.7}), concludes (\ref{eq3.3}), and completes the proof of Lemma \ref{lem3.1}.
\end{proof}

\section{Proofs of main results}
This section is devoted to proving Theorems \ref{thm1} and \ref{thm2}.

\subsection{Proof of Theorem  \ref{thm1}} To prove Theorem \ref{thm1}, we first recall the quantitative weighted result concerning with the sparse operator $\mathcal{A}_{\mathcal{S}}^{\beta,2}$ as follows.

\begin{lemma}{\rm(\cite{FH})} \label{lem4.1}
Let $0<\beta <n,\, 1< p<q<\infty$ with ${1}/{q}={1}/{p}-{\beta}/{n}$. If $\omega \in A_{p,q}$, then
\begin{equation}\label{eq4.1}
\| \mathcal{A}_{\mathcal{S}}^{\beta,2} f\|_{L^q(\omega^q)} \leq C(n) [\omega]_{A_{p,q}}^{\max\{\frac{p'}{q}(1-\frac{\beta}{n}), \frac{1}{2}-\frac{\beta}{n}\}} \|f\|_{L^p(\omega^p)}.
\end{equation}
\end{lemma}

Now we present the proof of Theorem \ref{thm1}.

\begin{proof}[Proof of Theorem \ref{thm1}] Note that (\ref{eq1.9}) directly follows from (\ref{eq1.7}), we prove only (\ref{eq1.8}).  By (\ref{eq3.1}) and (\ref{eq4.1}), we have
\begin{equation}\label{eq4.2}
	\|\widetilde{\mu}_{\Omega,\beta}^l f\|_{L^q(\omega^q)} \leq C(n,p) \|\Omega\|_{\infty} l\cdot [\omega]_{A_{p,q}}^{\max\{\frac{p'}{q}(1-\frac{\beta}{n}), \frac{1}{2}-\frac{\beta}{n}\}}\|f\|_{L^p(\omega^p)}.
\end{equation}
Recalling the definitions of $\widetilde{\mu}_{\Omega,\beta}^{l}$ ($l\in \mathbb{N}$), we can write $\widetilde{\mu}_{\Omega,\beta} $ as
$$\widetilde{\mu}_{\Omega,\beta} =\sum_{l=1}^{\infty}(\widetilde{\mu}_{\Omega,\beta}^{2^{l+1}} - \widetilde{\mu}_{\Omega,\beta}^{2^l} )+ \widetilde{\mu}_{\Omega,\beta}^{2}.$$
Take $\varepsilon:= c_n / [\omega]_{A_{p,q}}^{\max \{1,\frac{p'}{q}\}} $ with $c_n$ a constant depending only on $n$. It follows from \cite[Corollary 3.16]{HRT} that
$$[\omega^{1+\varepsilon}]_{A_{p,q}}\leq 4 [\omega]_{A_{p,q}}^{1+\varepsilon}.$$
Invoking the estimate (\ref{eq4.2}), we obtain
\begin{equation}\label{eq4.3}
\| \widetilde{\mu}_{\Omega,\beta}^{2^l} f-\widetilde{\mu}_{\Omega,\beta}^{2^{l+1}} f \|_{L^q(\omega^{(1+\varepsilon)q})} \leq C(n,p) \|\Omega\|_{\infty}2^l \cdot  [\omega]_{A_{p,q}}^{(1+\varepsilon)\max\{\frac{p'}{q}(1-\frac{\beta}{n}), \frac{1}{2}-\frac{\beta}{n}\}}\|f\|_{L^p(\omega^{(1+\varepsilon)p})}.	
\end{equation}
Also, by (\ref{2.11}), we know that
\begin{equation}\label{eq4.4}
\| \widetilde{\mu}_{\Omega,\beta}^{2^l} f-\widetilde{\mu}_{\Omega,\beta}^{2^{l+1}} f\|_{L^2(\mathbb{R}^n)} 	\leq  C(n)2^{-\frac{2^l}{16}} \|\Omega\|_{\infty}\|f\|_{L^{\frac{2n}{n+2\beta}}(\mathbb{R}^n)}.
\end{equation}
And taking $\omega=1$ in (\ref{eq4.3}), we get for $1<p,q<\infty$ and ${1}/{q}={1}/{p}-{\beta}/{n}$,
\begin{equation}\label{eq4.5}
	\| \widetilde{\mu}_{\Omega,\beta}^{2^l} f-\widetilde{\mu}_{\Omega,\beta}^{2^{l+1}} f \|_{L^q(\mathbb{R}^n)} \leq C(n,p) 2^l\|\Omega\|_{\infty}\|f\|_{L^p(\mathbb{R}^n)}.	
\end{equation}
Then by interpolating between (\ref{eq4.4}) and (\ref{eq4.5}), we get for $\rho \in (0,1)$,
\begin{equation}\label{eq4.6}
\| \widetilde{\mu}_{\Omega,\beta}^{2^l} f-\widetilde{\mu}_{\Omega,\beta}^{2^{l+1}} f \|_{L^q(\mathbb{R}^n)} \leq C(n,p) 2^l 2^{-\rho 2^l}\cdot \|\Omega\|_{\infty}\|f\|_{L^p(\mathbb{R}^n)}.	
\end{equation}

Moreover, applying the interpolation with change of measures (see \cite{SW}) between (\ref{eq4.4}) and (\ref{eq4.6}), we obtain
$$	\| \widetilde{\mu}_{\Omega,\beta}^{2^l} f-\widetilde{\mu}_{\Omega,\beta}^{2^{l+1}} f \|_{L^q(\omega^{q})} \leq C(n,p) 2^l 2^{-\rho \frac{\varepsilon}{1+\varepsilon}2^l} \cdot \|\Omega\|_{\infty}[\omega]_{A_{p,q}}^{\max\{\frac{p'}{q}(1-\frac{\beta}{n}), \frac{1}{2}-\frac{\beta}{n}\}}\|f\|_{L^p(\omega^p)}.	$$
A trivial computation (involving the inequlity $e^{-x}\leq 2x^{-2}$) shows that
$$\sum_{l=1} 2^l 2^{-\rho \frac{\varepsilon}{1+\varepsilon}2^l} \leq \sum_{l:2^l \leq \varepsilon^{-1}}2^l+ 2 \sum_{l: 2^l > \varepsilon^{-1}}2^l (\frac{1+\varepsilon}{2^l \varepsilon})^2\leq C(n) [\omega]_{A_{p,q}}^{\max \{1,\frac{p'}{q}\}}. $$
Consequently,
\begin{align*}
\|\widetilde{\mu}_{\Omega,\beta}f\|_{L^q(\omega^{q})}&\leq \sum_{l=1}^{\infty}\| \widetilde{\mu}_{\Omega,\beta}^{2^l} f-\widetilde{\mu}_{\Omega,\beta}^{2^{l+1}} f \|_{L^q(\omega^{q})} 	+ \|\widetilde{\mu}_{\Omega,\beta}^2f\|_{L^q(\omega^{q})}\\
& \leq C(n,p) \|\Omega\|_{\infty}[\omega]_{A_{p,q}}^{\max \{1,\frac{p'}{q}\}}[\omega]_{A_{p,q}}^{\max\{\frac{p'}{q}(1-\frac{\beta}{n}). \frac{1}{2}-\frac{\beta}{n}\}}\|f\|_{L^p(\omega^p)},
\end{align*}
This completes the proof of (\ref{eq1.8}). Theorem \ref{thm1} is proved.
\end{proof}

\subsection{Proof of Theorem \ref{thm2}}

In order to prove Theorem \ref{thm2}, we will use the following several lemmas.

\begin{lemma}{\rm(\cite{G})} \label{lem4.2}
If $\omega\in A_p, 1<p<\infty$, then for any $\varepsilon \in (0,1)$, $\omega^{\varepsilon} \in A_p$ and $[\omega^{\varepsilon}]_p\leq [\omega]_p^{\varepsilon}$.
\end{lemma}

\begin{lemma} {\rm (\cite{G})}\label{lem4.3}
Let $\omega\in A_p$ for some $1\leq p<\infty$. Then there exists constant $C$ and $r>1$ that depend only on the dimension $n,\,p$ and $[\omega]_{A_p}$ such that for every cube $Q$ we have
$$\Big(\frac{1}{|Q|}\int_{Q}\omega(t)^r dt \Big)^{{1}/{r}}\leq \frac{C}{|Q|}\int_{Q} \omega(t)dt, $$
where we fix any $0<\alpha <1$, then $r>1$ is chosen satisfying $$(\frac{2^n}{\alpha})^{r-1}\cdot (1-\frac{(1-\alpha)^p}{[\omega]_{p}}) <1$$
and
$$C=\Big[1+\frac{1}{(2^n \alpha^{-1})^{1-r}-(1-\frac{(1-\alpha)^p}{[\omega]_{p}}) }\Big]^{{1}/{r}}.$$
\end{lemma}

\begin{lemma}{\rm (\cite{G})} \label{lem4.4}{\rm(J-N inequality)} There exists constants $C_1,\,C_2>0$
such that for any $b \in BMO(\mathbb{R}^n)$,
$$\sup_{Q}\frac{1}{|Q|}\int_{Q}\exp \Big(\frac{C_2}{\|b\|_*}|b(x)-b_Q|\Big)dx\leq C_1,$$
where the constants $C_1,C_2$ not depend on $b$ and $Q$.   	
\end{lemma}

\begin{lemma}\label{lem4.5}
Let $0<\beta<n$, $1<p<{n}/{\beta}$ and ${1}/{q}={1}/{p}-{\beta}/{n}$. If $b \in BMO(\mathbb{R}^n)$, then for any $\theta \in [0,2\pi]$, there exists $\lambda=\frac{C(n,p,\beta)}{\|b\|_*}$ satisfying $e^{\lambda b \cos \theta } \in A_{p,q}$ and
$$[e^{\lambda b \cos \theta } ]_{A_{p,q}}\leq C_1^2,$$
where $C_1$ is a constant not depending on $\theta,\beta,b$. 	
\end{lemma}
\begin{proof}
Let $\widetilde{b}=:b\cos \theta$ and it is obviously that $\|\widetilde{b}\|_*\leq \|b\|_*$.	
Firstly, we shall prove that there exists $\lambda_0=\frac{C(n,p,\beta)}{\|b\|_*}$ such that
$e^{\lambda_0 b \cos \theta} \in A_{\frac{q(n-\beta)}{n}}$
and
$$[e^{\lambda_0 b \cos \theta } ]_{\frac{q(n-\beta)}{n}}\leq C_1^2.$$
Applying Lemma \ref{lem4.4} to $\widetilde{b}$, there exists  constants $C_1,C_2>0$ not depending on $\widetilde{b},\theta$ such that for any $0<\lambda_0<\frac{C_2}{\|\widetilde{b}\|_*}$,
\begin{equation}\label{eq4.7}
\sup_{Q}\frac{1}{|Q|}\int_{Q} e^{\lambda_0|\widetilde{b}(x)-\widetilde{b}_Q|}dx\leq C_1.
\end{equation}

Now we consider the following two cases.
If $\frac{q(n-\beta)}{n}\ge 2$, let $\lambda_0= \frac{C_2}{\|b\|_*}  \leq\frac{C_2}{\|\widetilde{b}\|_*}$. Then by (\ref{eq4.7}), for any $\theta \in [0,2\pi]$,
\begin{align*}
&\sup_{Q}\Big(\frac{1}{|Q|}\int_{Q}e^{\lambda_0 \widetilde{b}(x) }dx\Big)	\Big(\frac{1}{|Q|}\int_{Q}e^{-\lambda_0 \widetilde{b}(x) }dx\Big)	\\
=& \sup_{Q}\Big(\frac{1}{|Q|}\int_{Q}e^{\lambda_0 (\widetilde{b}(x)-\widetilde{b}_Q) }dx\Big)	\Big(\frac{1}{|Q|}\int_{Q}e^{\lambda_0 (\widetilde{b}_Q-\widetilde{b}(x)) }dx\Big)\\
\leq & \sup_{Q}\Big(\frac{1}{|Q|}\int_{Q} e^{\lambda_0|\widetilde{b}(x)-\widetilde{b}_Q|}dx\Big)^2\\
\leq & C_1^2.
\end{align*}
Hence, $e^{\lambda_0 b \cos \theta} \in A_2 \subset A_{\frac{q(n-\beta)}{n}}$ and $[e^{\lambda_0 b \cos \theta} ]_{\frac{q(n-\beta)}{n}}\leq C_1^2$.

If $1<\frac{q(n-\beta)}{n}< 2$, then $(\frac{q(n-\beta)}{n} )' \ge 2$. Let $\widetilde{\lambda}_0= \frac{C_2}{\|b\|_*}  \leq\frac{C_2}{\|\widetilde{b}\|_*}$. Repeat the above step of $\frac{q(n-\beta)}{n}\ge 2$, we can get
$e^{-\widetilde{\lambda}_0 b \cos \theta} \in A_2 \subset A_{(\frac{q(n-\beta)}{n})'}$ and $[e^{-\widetilde{\lambda}_0 b \cos \theta} ]_{(\frac{q(n-\beta)}{n})'}\leq C_1^2$.
According to the $A_p$ weight's property, it is easy to get that $(e^{-\widetilde{\lambda}_0 b \cos \theta })^{1-\frac{q(n-\beta)}{n}}\in A_{\frac{q(n-\beta)}{n}}$ and
$$[e^{-\widetilde{\lambda}_0(1-\frac{q(n-\beta)}{n}) b \cos \theta }]_{\frac{q(n-\beta)}{n}}= [e^{-\widetilde{\lambda}_0 b \cos \theta} ]_{(\frac{q(n-\beta)}{n})'}\leq C_1^2 .$$
Let $\lambda_0=\frac{C_2}{\|b\|_*}(\frac{q(n-\beta)}{n}-1)$. Hence,
 $[e^{\lambda_0 b \cos \theta} ]_{\frac{q(n-\beta)}{n}}\leq C_1^2.$

Combining with the above two cases,
there exists $\lambda_0=\frac{C(n,q,\beta)}{\|b\|_*}$ such that
 $$[e^{\lambda_0 b \cos \theta} ]_{\frac{q(n-\beta)}{n}}\leq C_1^2.$$

Finally, let $\lambda=\frac{\lambda_0}{q}=\frac{C(n,q,\beta)}{\|b\|_*}$ and we can obtain
$$[e^{\lambda b \cos \theta}]_{A_{p,q}}=[e^{\lambda b \cos \theta q}]\leq C_1^2.$$
Therefore, we complete the proof of Lemma \ref{lem4.5}.
\end{proof}

Now we are in the position to prove Theorem \ref{thm2}.

\begin{proof}[Proof of Theorem \ref{thm2}] This proof will be based on the Cauchy integral formula. For $z \in \mathbb{C}$,
$$g(z)=e^{z[b(x)-b(y)]}$$
is analytic on $\mathbb{C}$. Thus by the Cauchy formula we get for any $\varepsilon>0$,
\begin{align*}
b(x)-b(y)=g'(0)&=\frac{1}{2 \pi i}\int_{|z|=\varepsilon}\frac{g(z)}{z^2}dz
=\frac{1}{2 \pi \varepsilon}\int_0^{2 \pi}e^{\varepsilon e^{i \theta}[b(x)-b(y)]}e^{-i \theta}d\theta.	
\end{align*}
Applying the above formula and Minkowski's inequality, it implies that
\begin{align*}
\mu_{\Omega,\beta}^b f(x)&=\mu_{\Omega,\beta}((b(x)-b(\cdot )f(\cdot))(x)\\
&= \frac{1}{2 \pi \varepsilon}\Big(\int_{0}^{\infty}\Big|\int_{0}^{2 \pi}\int_{|x-y|\leq t} \frac{\Omega(x-y)}{|x-y|^{n-\beta -1}}  e^{\varepsilon e^{i \theta}[b(x)-b(y)]}f(y)dye^{-i \theta}d\theta \Big|^2\frac{dt}{t^3}\Big)^{1/2}\\
& \leq \frac{1}{2 \pi \varepsilon}\int_0^{2 \pi} \Big( \int_{0}^{\infty } \Big| \int_{\left | x-y \right |\le t }^{}  \frac{\Omega (x-y)}{\left | x-y \right |^{n-1-\beta }  } e^{-\varepsilon e^{i \theta}b(y)} f(y)dy   \Big| ^{2}\frac{dt}{t^3} \Big)^{{1}/{2} } e^{\varepsilon \cos\theta b(x)}  d\theta\\
& =: \frac{1}{2 \pi \varepsilon}\int_0^{2 \pi} \mu_{\Omega,\beta}(h_{\theta})(x) e^{\varepsilon \cos\theta b(x)} d\theta,
\end{align*}
where $h_{\theta}(x)=f(x) e^{\varepsilon e^{i \theta}b(x)}$ for $\theta \in [0,2\pi]$.
Then, using the Minkowski's inequality, we have for $ \omega\in A_{p,q}$,
\begin{equation}\label{eq4.8}
\begin{aligned}
\|\mu_{\Omega,\beta}^b f\|_{L^q(\omega^q)}
& \leq \Big(\int_{\mathbb{R}^n}\Big| \frac{1}{2 \pi \varepsilon}\int_0^{2 \pi} \mu_{\Omega,\beta}(h_{\theta})(x) e^{\varepsilon \cos\theta b(x)} d\theta\Big|^q \omega^q(x)dx\Big)^{1/q}\\
& \leq \frac{1}{2 \pi \varepsilon}\int_0^{2 \pi}  \Big(\int_{\mathbb{R}^n} \mu_{\Omega,\beta}(h_{\theta})(x)^q e^{q\varepsilon \cos\theta b(x)} \omega^q(x)dx\Big)^{1/q} d\theta\\
& = \frac{1}{2 \pi \varepsilon}\int_0^{2 \pi} \| \mu_{\Omega,\beta}(h_{\theta})\|_{L^q(\omega^q e^{q\varepsilon \cos\theta b} )}d \theta.
\end{aligned}	
\end{equation}
 Note that for $f \in L^p(\omega^p)$, it is easy to check that for any $\theta \in [0,2\pi]$,
 $$h_{\theta} \in L^{p}(\omega^p e^{pb(\cdot)\varepsilon \cos\theta}) \,\,\mathrm{and}\,\, \|h_{\theta}\|_{L^{p}(\omega^p e^{pb(\cdot)\varepsilon \cos\theta})}=\|f\|_{L^p(\omega^p)}.$$
We assume the following inequality to be true:
\begin{equation}\label{eq4.9}
[\omega e^{b(\cdot)\varepsilon \cos\theta}]_{A_{p,q}}\leq 49 \cdot C_1^2 [\omega]_{A_{p,q}},
\end{equation}
where the constant $C_1$ is not depending on $\varepsilon,\theta$ and $\varepsilon$ is be chosen later. Then by (\ref{eq1.8}), (\ref{eq4.8}) and (\ref{eq4.9}), we can obtain for $1<p<\infty$, $\omega\in A_{p,q}$ and $0<\beta<1/2$,
\begin{equation}\label{eq4.10}
\begin{aligned}
\|\mu_{\Omega,\beta}^b f\|_{L^q(\omega^q)}\leq & C(n,p) \|\Omega\|_{\infty}[\omega]_{A_{p,q}}^{\max\{1,\frac{p'}{q}\}} [\omega]_{A_{p,q}}^{\max\{\frac{p'}{q}(1-\frac{\beta}{n}),\frac{1}{2}-\frac{\beta}{n}\}}	\\
&\times \frac{1}{2 \pi \varepsilon}\int_0^{2 \pi} \|h_{\theta}\|_{L^p(\omega^p e^{p\varepsilon \cos\theta b} )}d \theta\\
\leq &  C(n,p) \|\Omega\|_{\infty} \|b\|_{*}[\omega]_{A_{p,q}}^{1+\max\{1,\frac{p'}{q}\}} [\omega]_{A_{p,q}}^{\max\{\frac{p'}{q}(1-\frac{\beta}{n}),\frac{1}{2}-\frac{\beta}{n}\}}\|f\|_{L^p(\omega^p)}.
\end{aligned}
\end{equation}
This is the desired conclusion.

It remains to prove (\ref{eq4.9}). Recall the facts that
$$[\omega^q]_{1+q/p'}=[\omega]_{A_{p,q}}\,\, \mathrm{and }\,\, [\omega^{-p'}]_{1+p'/q}=[\omega]_{A_{p,q}}^{p'/q}.$$
By Lemma \ref{lem4.2}, we can get $w \in A_{1+q/p'}$ and $[\omega]_{1+q/p'}\leq [\omega]_{A_{p,q}}$.
Let $[\widetilde{w}]=[\omega]_{A_{p,q}}^{\max\{1,\frac{p'}{q}\}}$. Fix $\alpha \in (0,1)$ and take
$$r=1-\frac{\log (1-(1-\alpha)^{1+\frac{q}{p'}+\frac{p'}{q}} [\widetilde{\omega}]^{-1})}{2\log(2^n \alpha^{-1})}.$$
Then it is easy to check that
\begin{align*}
& \Big(\frac{2^n}{\alpha}\Big)^{r-1}\cdot \Big(1-\frac{(1-\alpha)^{1+q/p'}}{[\omega]_{1+q/p'}}\Big) \\
&\qquad=\Big(1-\frac{(1-\alpha)^{1+q/p'+p'/q}}{[\widetilde{\omega}]}\Big)^{-{1}/{2}}\cdot\Big(1-\frac{(1-\alpha)^{1+q/p'}}{[\omega]_{1+q/p'}}\Big)\\
&\qquad\leq  \Big(1-\frac{(1-\alpha)^{1+q/p'}}{[\omega]_{1+q/p'}}\Big)^{-{1}/{2}} \cdot \Big(1-\frac{(1-\alpha)^{1+q/p'}}{[\omega]_{1+q/p'}}\Big)\\
&\qquad=\Big(1-\frac{(1-\alpha)^{1+q/p'}}{[\omega]_{1+q/p'}}\Big)^{{1}/{2}}<1,
\end{align*}
and
\begin{align*}
& \Big(\frac{2^n}{\alpha}\Big)^{r-1}\cdot \Big(1-\frac{(1-\alpha)^{1+p'/q}}{[\omega^{-p'}]_{1+p'/q}}\Big) \\
&\qquad=\Big(1-\frac{(1-\alpha)^{1+q/p'+p'/q}}{[\widetilde{w}]}\Big)^{-{1}/{2}}\cdot \Big(1-\frac{(1-\alpha)^{1+p'/q}}{[\omega^{-p'}]_{1+p'/q}}\Big)\\
&\qquad\leq \Big(1-\frac{(1-\alpha)^{1+p'/q}}{[\omega^{-p'}]_{1+p'/q}}\Big)^{-{1}/{2}}\cdot \Big(1-\frac{(1-\alpha)^{1+p'/q}}{[\omega^{-p'}]_{1+p'/q}}\Big)\\
&\qquad= \Big(1-\frac{(1-\alpha)^{1+p'/q}}{[\omega^{-p'}]_{1+p'/q}}\Big)^{{1}/{2}}<1.
\end{align*}
Applying Lemma \ref{lem4.3} to the weights $w \in A_{1+q/p'}$ and $\omega^{-p'} \in A_{1+p'/q}$ and letting $r=1-\frac{\log (1-(1-\alpha)^{1+\frac{q}{p'}+\frac{p'}{q}} [\widetilde{w}]^{-1})}{2\log(2^n \alpha^{-1})}$, then we can obtain
\begin{equation} \label{eq4.11}
\Big(\frac{1}{|Q|}\int_{Q}w(t)^r dt \Big)^{{1}/{r}}\leq \frac{C_1}{|Q|}\int_{Q} w(t)dt,
\end{equation}
and
\begin{equation}\label{eq4.12}
\Big(\frac{1}{|Q|}\int_{Q}w(t)^{-p'r} dt \Big)^{{1}/{r}}\leq \frac{C_2}{|Q|}\int_{Q} w(t)^{-p'}dt,
\end{equation}
where
$$C_1=\Big[1+\frac{1}{(2^n \alpha^{-1})^{1-r}-(1-\frac{(1-\alpha)^{1+q/p'}}{[\omega]_{1+q/p'}}) }\Big]^{{1}/{r}} ,$$
and
$$C_2=\Big[1+\frac{1}{(2^n \alpha^{-1})^{1-r}-(1-\frac{(1-\alpha)^{1+p'/q}}{[\omega^{-p'}]_{1+p'/q}}) }\Big]^{{1}/{r}} .$$
Write $\frac{(1-\alpha)^{1+q/p'}}{[\omega]_{1+q/p'}} ,\frac{(1-\alpha)^{1+p'/q}}{[\omega^{-p'}]_{1+p'/q}} $ as $\beta_1,\beta_2$, respectively. We can pick $0<\alpha<1$ satisfying the following two conditions:
$$\frac{1}{16}\leq  \frac{(1-\alpha)^{1+p'/q+q/p'}}{[\widetilde{w}]}\leq \beta_1=\frac{(1-\alpha)^{1+q/p'}}{[\omega]_{1+q/p'}} \leq \frac{1}{4},$$
and
$$\frac{1}{16}\leq \frac{(1-\alpha)^{1+p'/q+q/p'}}{[\widetilde{w}]}\leq \beta_2=\frac{(1-\alpha)^{1+p'/q}}{[\omega^{-p'}]_{1+p'/q}} \leq \frac{1}{4}.$$
Then the constant $C_1$ is controlled by
\begin{align*}
C_1&=\Big[1+\frac{1}{(2^n \alpha^{-1})^{1-r}-\beta_1}\Big]^{1/r}\\
& \leq 	\Big[1+\frac{1}{{(\beta_1}^{1/2}-\beta_1)}\Big]^{1/r}\\
& \leq 1+\frac{1}{{(\beta_1}^{1/2}-\beta_1)}\\
&=1+\frac{1}{\beta_1^{1/2}}+\frac{1}{1-\beta_1^{1/2}}\\
&\leq 1+4+2=7,
\end{align*}
where we use $1/16<\beta_1<1/4$ in the last inequality. Similarly, the constant $C_2$ is also controlled by the constant $7$.

Using the H\"older inequality and by (\ref{eq4.11}), (\ref{eq4.12}), we have
\begin{equation}\label{eq4.13}
\begin{split}
&[e^{b\varepsilon \cos \theta}\omega]_{A_{p,q}}\\
&= \sup_{Q}\Big(\frac{1}{|Q|}\int_{Q}e^{qb(x) \varepsilon \cos \theta}\omega^q(x)dx\Big)\Big(\frac{1}{|Q|}\int_{Q} e^{-p'b(x) \varepsilon \cos \theta}\omega^{-p'}(x)dx\Big)^{q/p'}\\
& \leq \sup_{Q} \Big(\frac{1}{|Q|}\int_{Q}\omega^{qr}(x)dx\Big)^{{1}/{r}} \Big(\frac{1}{|Q|}\int_{Q}e^{qr'b(x) \varepsilon \cos \theta}dx\Big)^{{1}/{r'}}\\
& \qquad\quad\times  \Big(\frac{1}{|Q|}\int_{Q}\omega^{-rp'}(x)dx\Big)^{{1}/{r}} \Big(\frac{1}{|Q|}\int_{Q}e^{-p'r'b(x) \varepsilon \cos \theta}dx\Big)^{{1}/{r'}}\\
&\leq 49\cdot
\sup_{Q}\Big(\frac{1}{|Q|}\int_{Q}\omega^q(x)dx\Big)\Big(\frac{1}{|Q|}\int_{Q}\omega^{-p'}(x)dx\Big)^{q/p'}\\
&\qquad\quad\times \sup_{Q} \Big(\frac{1}{|Q|}\int_{Q}e^{qr'b(x) \varepsilon \cos \theta}dx\Big)^{{1}/{r'}} \Big(\frac{1}{|Q|}\int_{Q}e^{-p'r'b(x) \varepsilon \cos \theta}dx\Big)^{{1}/{r'}}\\
&= 49 \cdot [\omega]_{A_{p,q}}[e^{r'b\varepsilon \cos \theta}]_{A_{p,q}}^{{1}/{r'}}\\
&\leq 49 \cdot [\omega]_{A_{p,q}}[e^{r'b\varepsilon \cos \theta}]_{A_{p,q}}.
\end{split}
\end{equation}
And by Lemma \ref{lem4.5}, taking $\varepsilon=\frac{C(n,q,\beta)}{r'\|b\|_*}$, we conclude that
$$[e^{r'\varepsilon b \cos \theta}]_{A_{p,q}}=[e^{\lambda b \cos \theta}]_{A_{p,q}}\leq C_1^2.$$
This, together with (\ref{eq4.13}), implies that (\ref{eq4.8}) holds and completes the proof of (\ref{eq1.10}).

Finally, by (\ref{eq1.9}) and (\ref{eq4.8}), employing the same arguments in proving (\ref{eq4.10}), we can obtain (\ref{eq1.11}) and completes the proof of Theorem \ref{thm2}.
\end{proof}

\end{document}